\documentclass[12pt]{amsart}

\pdfoutput=1

\usepackage[colorlinks=true]{hyperref}
\usepackage{amssymb}
\usepackage[all]{xy}
\usepackage{eucal}

\usepackage[backend=bibtex,
            isbn=false,
            doi=false,
            style=alphabetic,
            giveninits=true,
            citestyle=alphabetic,
            maxnames=5]{biblatex}
\usepackage{url}
\renewbibmacro{in:}{}
\bibliography{additivity.bib}

\usepackage{tikz}
\usetikzlibrary{calc,matrix}
\tikzset{w/.style={circle, draw,inner sep=1pt},b/.style={circle,draw,fill,inner sep=2pt}, s/.style={rectangle, draw,inner sep=3pt}}

\newtheorem{thm}{Theorem}[section]
\newtheorem{prop}[thm]{Proposition}
\newtheorem{lm}[thm]{Lemma}
\newtheorem{cor}[thm]{Corollary}
\newtheorem{conjecture}[thm]{Conjecture}
\newtheorem*{theorem}{Theorem}
\theoremstyle{definition}
\newtheorem{defn}[thm]{Definition}
\theoremstyle{remark}
\newtheorem{remark}[thm]{Remark}
\newtheorem{example}[thm]{Example}

\newcommand{\cA}{\mathcal{A}}
\newcommand{\cB}{\mathcal{B}}
\newcommand{\cC}{\mathcal{C}}
\newcommand{\cD}{\mathcal{D}}

\newcommand{\cO}{\mathcal{O}}

\newcommand{\B}{\mathrm{B}}
\newcommand{\N}{\mathrm{N}}
\newcommand{\T}{\mathrm{T}}

\newcommand{\bu}{\mathbf{1}}
\newcommand{\bL}{\mathbb{L}}

\newcommand{\g}{\mathfrak{g}}
\newcommand{\h}{\mathfrak{h}}
\renewcommand{\k}{\mathfrak{k}}
\newcommand{\s}{\mathbf{s}}

\newcommand{\Ass}{\mathrm{Ass}}
\newcommand{\coAss}{\mathrm{coAss}}
\newcommand{\Br}{\mathrm{Br}}
\newcommand{\Comm}{\mathrm{Comm}}
\newcommand{\coComm}{\mathrm{coComm}}
\newcommand{\E}{\mathbb{E}}
\newcommand{\coE}{\mathrm{co}\mathbb{E}}
\newcommand{\Lie}{\mathrm{Lie}}
\newcommand{\coLie}{\mathrm{coLie}}
\renewcommand{\P}{\mathbb{P}}
\newcommand{\coP}{\mathrm{co}\mathbb{P}}

\newcommand{\aug}{\mathrm{aug}}

\newcommand{\coaug}{\mathrm{coaug}}
\newcommand{\cu}{\mathrm{cu}}
\newcommand{\kos}{W_{\mathrm{Kos}}}
\newcommand{\qis}{W_{\mathrm{qis}}}
\newcommand{\sgn}{\mathrm{sgn}}
\newcommand{\str}{\mathrm{str}}
\newcommand{\un}{\mathrm{un}}

\newcommand{\alg}{\mathrm{Alg}}
\newcommand{\ialg}{\mathcal{A}\mathrm{lg}}
\newcommand{\bialg}{\mathrm{BiAlg}}
\newcommand{\coalg}{\mathrm{CoAlg}}

\newcommand{\lmod}{\mathrm{LMod}}
\newcommand{\ilmod}{\mathcal{L}\mathrm{Mod}}

\newcommand{\Ch}{\mathrm{Ch}}
\newcommand{\iCh}{\mathcal{C}\mathrm{h}}

\newcommand{\add}{\mathrm{add}}
\newcommand{\ad}{\mathrm{ad}}
\newcommand{\Arr}{\mathrm{Arr}}
\renewcommand{\Bar}{\mathrm{Bar}}
\newcommand{\C}{\mathrm{C}}
\newcommand{\Cn}{\mathrm{Cn}}
\newcommand{\Cois}{\mathrm{Cois}}
\renewcommand{\d}{\mathrm{d}}
\newcommand{\Der}{\mathrm{Der}}
\newcommand{\forget}{\mathrm{forget}}
\newcommand{\Fun}{\mathrm{Fun}}
\newcommand{\Hom}{\mathrm{Hom}}
\newcommand{\ish}{\textrm{!`}}
\newcommand{\id}{\mathrm{id}}
\newcommand{\U}{\mathrm{U}}
\newcommand{\triv}{\mathrm{triv}}
\newcommand{\Z}{\mathrm{Z}}

\DeclareMathOperator{\Spec}{Spec}
\DeclareMathOperator{\Sym}{Sym}

\newcommand{\defterm}[1]{\textbf{\emph{#1}}}

\newcommand{\adj}[2]{
\xymatrix{
#1 \ar@<.5ex>[r] & #2 \ar@<.5ex>[l]
}
}
\renewcommand{\to}{\longrightarrow}

\begin{document}
\title{Braces and Poisson additivity}
\author{Pavel Safronov}
\address{Max-Planck-Institut F\"{u}r Mathematik, Bonn, Germany}
\curraddr{University of Z\"{u}rich, Winterthurerstrasse 190, 8057 Z\"{u}rich, Switzerland}
\email{pavel.safronov@math.uzh.ch}
\begin{abstract}
We relate the brace construction introduced by Calaque and Willwacher to an additivity functor. That is, we construct a functor from brace algebras associated to an operad $\cO$ to associative algebras in the category of homotopy $\cO$-algebras. As an example, we identify the category of $\P_{n+1}$-algebras with the category of associative algebras in $\P_n$-algebras. We also show that under this identification there is an equivalence of two definitions of derived coisotropic structures in the literature.
\end{abstract}
\maketitle

\section*{Introduction}

This paper is devoted to a proof of an equivalence of symmetric monoidal $\infty$-categories
\[\ialg_{\P_{n+1}}\cong \ialg(\ialg_{\P_n})\]
between the $\infty$-category of $\P_{n+1}$-algebras and the $\infty$-category of associative algebras in the $\infty$-category of $\P_n$-algebras. Here $\P_{n+1}$ is the operad which controls dg commutative algebras together with a Poisson bracket of degree $-n$.

\subsection*{Braces}

Let $\cO$ be a dg operad and $\cC$ its Koszul dual cooperad which is assumed to be Hopf. Following Tamarkin's work \cite{Ta1} on the deformation complex of a $\P_n$-algebra, Calaque and Willwacher \cite{CW} introduced an operad $\Br_{\cC}$ of brace algebras which acts on the deformation complex of any homotopy $\cO$-algebra. Moreover, they have remarked that the brace construction is an analogue of the Boardman--Vogt tensor product $\E_1\otimes \Omega\cC$ of operads. This is suggested by the following examples:
\begin{itemize}
\item If $\bu$ is the trivial cooperad, $\Br_{\bu}\cong \E_1$.
\item If $\coAss$ is the cooperad of coassociative coalgebras, $\Br_{\coAss}\{1\}\cong \E_2$.
\item If $\coComm$ is the cooperad of cocommutative coalgebras, $\Br_{\coComm}\cong \Lie$.
\item If $\coP_n$ is the cooperad of $\P_n$-coalgebras, $\Br_{\coP_n}\{n\}\cong \P_{n+1}$.
\end{itemize}

In this paper we explain to what extent this is true. Namely, suppose $\cC$ is a Hopf cooperad satisfying a minor technical assumption. We construct a functor of $\infty$-categories
\begin{equation}
\ialg_{\Br_{\cC}}\to \ialg(\ialg_{\cO})
\label{eq:braceadditivityfunctor}
\end{equation}
from the $\infty$-category of $\Br_{\cC}$-algebras to the $\infty$-category of associative algebras in the $\infty$-category of $\cO$-algebras. Let us note that we did not assume that $\cO$ is a Hopf operad and the symmetric monoidal structure on $\ialg_{\cO}$ comes from the Koszul dual side.

Unfortunately, we do not know if \eqref{eq:braceadditivityfunctor} is an equivalence in general, but we do show that it is an equivalence in two examples of interest: namely, Lie algebras and Poisson algebras.

Suppose $\cC=\coComm$. As we have mentioned, $\Br_{\coComm}\cong\Lie$ and so we get a functor
\[\add\colon\ialg_{\Lie}\to \ialg(\ialg_{\Lie}).\]
We show that it is an equivalence and in fact coincides with the functor which sends a Lie algebra $\g$ to the associative algebra object in the category of Lie algebras $0\times_{\g} 0$ (see Proposition \ref{prop:LieadditivityOmega}). We also show that the same functor can be constructed as follows. Given a Lie algebra $\g$, the universal enveloping algebra $\U(\g)$ is a cocommutative bialgebra, i.e. an associative algebra object in cocommutative coalgebras. Identifying cocommutative coalgebras with Lie algebras using Koszul duality we obtain the same functor (see Proposition \ref{prop:LieadditivityU}). Let us mention that the underlying Lie algebra structure on $\add(\g)$ is canonically trivial by Proposition \ref{prop:underlyingLietrivial}.

Note that $\Br_{\coComm}$ is an important operad in itself and appears for instance in the description of the Atiyah bracket of vector fields, see Section \ref{sect:brcocomm}.

\subsection*{Poisson additivity}
The additivity functor is more interesting in the case of $\P_{n+1}$-algebras. So, take $\cC=\coP_n$. Since $\Br_{\coP_n}\{n\}\cong \P_{n+1}$, we obtain a functor
\[\add\colon\ialg_{\P_{n+1}}\to \ialg(\ialg_{\P_n}).\]

The following statement combines Propositions \ref{prop:CommPoissoncompatible} - \ref{prop:LiePoissonForgetcompatible} and Theorem \ref{thm:Poissonadditivity}.

\begin{theorem}
The additivity functor
\[\ialg_{\P_{n+1}}\to \ialg(\ialg_{\P_n})\]
is an equivalence of symmetric monoidal $\infty$-categories.

Moreover, the diagrams
\[
\xymatrix{
\ialg_{\P_{n+1}} \ar[d] \ar[r] & \ialg(\ialg_{\P_n}) \ar[d] \\
\ialg_{\Comm} & \ialg_{\P_n} \ar[l]
}
\]
and
\[
\xymatrix{
\ialg_{\P_{n+1}} \ar[r] \ar@<.5ex>[d] & \ialg(\ialg_{\P_n}) \ar@<.5ex>[d] \\
\ialg_{\Lie} \ar[r] \ar@<.5ex>^{\Sym}[u] & \ialg(\ialg_{\Lie}) \ar@<.5ex>^{\Sym}[u]
}
\]
commute.
\end{theorem}
Nick Rozenblyum has given an independent proof of this result in the language of factorization algebras. This statement is a Poisson version of the additivity theorem \cite[Theorem 5.1.2.2]{HA} for $\E_n$-algebras proved by Dunn and Lurie: one has an equivalence
\[\ialg_{\E_{n+1}}\cong \ialg(\ialg_{\E_n})\]
of symmetric monoidal $\infty$-categories, where $\E_n$ is the operad of little $n$-disks.

One has the following explicit description of the additivity functor for Poisson algebras which uses some ideas of Tamarkin (see \cite{Ta1} and \cite{Ta2}). For simplicity, we describe the construction in the case of non-unital $\P_{n+1}$-algebras. In the case of unital $\P_{n+1}$-algebras one has to take care of the natural curving appearing on the Koszul dual side, but otherwise the construction is identical (see Section \ref{sect:unitalpoissonadditivity}). If $A$ is a commutative algebra, we can consider its Harrison complex $\coLie(A[1])$ which is a Lie coalgebra. If $A$ is moreover a $\P_{n+1}$-algebra, then the Harrison complex $\coLie(A[1])[n-1]$ has a natural structure of an $(n-1)$-shifted Lie bialgebra (Definition \ref{def:liebialg}) which defines a functor
\[\alg_{\P_{n+1}}\to \bialg_{\Lie_{n-1}}.\]
Given a Lie algebra $\g$, its universal enveloping algebra $\U(\g)$ is a cocommutative bialgebra. If $\g$ is moreover an $(n-1)$-shifted Lie bialgebra, then $\U(\g)$ acquires a natural cobracket making it into an associative algebra object in $\P_n$-coalgebras. Thus, we get a functor
\[\bialg_{\Lie_{n-1}}\to \alg(\coalg_{\coP_n}).\]
Applying Koszul duality we identify $\coalg_{\coP_n}$ with the category of $\P_n$-algebras thus giving the required additivity functor.

An important point we have neglected in this discussion is that at the very end one has to pass from the localization of the category of associative algebras in $\P_n$-coalgebras to the $\infty$-category of (homotopy) associative algebras in the localization of the category of $\P_n$-coalgebras. That is, we have a natural functor
\[\alg(\coalg_{\coP_n})[\kos^{-1}]\to \ialg(\coalg_{\coP_n}[\kos^{-1}])\]
where $\kos$ is a certain natural class of weak equivalences we define in the paper, $\alg$ is the 1-category of associative algebras and $\ialg$ is the $\infty$-category of (homotopy) associative algebras. The fact that this functor is an equivalence is not automatic: the corresponding rectification statement was proved in \cite[Theorem 4.1.8.4]{HA} under the assumption that the model category in question is a monoidal model category while the monoidal structure on $\coalg_{\coP_n}$ does not even preserve colimits. Furthermore, if we do not pass to the Koszul dual side, the localization functor (where $\qis$ is the class of quasi-isomorphisms)
\[\alg(\alg_{\P_n})[\qis^{-1}]\to \ialg(\alg_{\P_n}[\qis^{-1}])\]
is \emph{not} an equivalence. Indeed, $\alg(\alg_{\P_n})$ is equivalent to the category of commutative algebras while we prove that $\ialg(\alg_{\P_n}[\qis^{-1}])$ is equivalent to the $\infty$-category of $\P_{n+1}$-algebras.

Let us note that the underlying commutative structure on $\add(A)$ for a $\P_{n+1}$-algebra $A$ coincides with the commutative structure on $A$ and the underlying Lie structure on $\add(A)$ is trivial. However, the underlying $\P_n$-structure on $\add(A)$ is not necessarily commutative.

One motivation for developing Poisson additivity is the recent work of Costello and Gwilliam \cite{CG} that formalizes algebras of observables in quantum field theories. In that work a topological quantum field theory is described by a locally-constant factorization algebra on the spacetime manifold valued in $\E_0$-algebras. Since locally-constant factorization algebras on $\mathbb{R}^n$ are the same as $\E_n$-algebras, we see that observables in an $n$-dimensional topological quantum field theory are described by $\E_n\otimes \E_0=\E_n$-algebras. Similarly, classical topological field theories are described by locally-constant factorization algebras valued in $\P_0$-algebras, which in the case of $\mathbb{R}^n$ are the same as $\E_n$-algebras in $\P_0$-algebras. Our result thus shows that observables in an $n$-dimensional classical topological field theory are described by a $\P_n$-algebra (a natural result one expects by extrapolating from the case of topological quantum mechanics which is $n=1$).

\subsection*{Coisotropic structures}

Another motivation is given by the theory of shifted Poisson geometry developed by Calaque--Pantev--To\"{e}n--Vaqui\'{e}--Vezzosi and, more precisely, derived coisotropic structures. Recall that an $n$-shifted Poisson structure on an affine scheme $\Spec A$ for $A$ a commutative dg algebra is described by a $\P_{n+1}$-algebra structure on $A$. Now suppose $f\colon \Spec B\rightarrow \Spec A$ is a morphism of affine schemes. In \cite{CPTVV} the following notion of derived coisotropic structures was introduced. Assume the statement of Poisson additivity. Then one can realize $A$ as an associative algebra in $\P_n$-algebras and a coisotropic structure on $f$ is a lift of the natural action of $A$ on $B$ in commutative algebras to $\P_n$-algebras. Let us denote by $\Cois^{CPTVV}(f, n)$ the space of such coisotropic structures.

A more explicit definition of derived coisotropic structures was given in \cite{Sa} and \cite{MS} which does not rely on Poisson additivity. An action of a $\P_{n+1}$-algebra $A$ on a $\P_n$-algebra $B$ was modeled by a certain colored operad $\P_{[n+1, n]}$ and a derived coisotropic structure was defined to be the lift of the natural action of $A$ on $B$ in commutative algebras to an algebra over the operad $\P_{[n+1, n]}$. Let us denote by $\Cois^{MS}(f, n)$ the space of such coisotropic structures.

In this paper we show that these two notions coincide. The following statement is Corollary \ref{cor:coisotropicadditivity}.

\begin{theorem}
Suppose $f\colon A\rightarrow B$ is a morphism of commutative dg algebras. One has a natural equivalence
\[\Cois^{MS}(f, n)\cong \Cois^{CPTVV}(f, n)\]
of spaces of $n$-shifted coisotropic structures.
\end{theorem}

This statement is proved by developing a relative analogue of the Poisson additivity functor. Namely, Theorem \ref{thm:relPoissonadditivity} asserts that the $\infty$-category of $\P_{[n+1, n]}$-algebras is equivalent to the $\infty$-category of pairs $(A, M)$, where $A$ is an associative algebra and $M$ is an $A$-module in the $\infty$-category of $\P_n$-algebras.

\subsection*{Notations}

\begin{itemize}
\item Given a relative category $(\cC, W)$ we denote by $\cC[W^{-1}]$ the underlying $\infty$-category.

\item We work over a field $k$ of characteristic zero; $\Ch$ denotes the category of chain complexes of $k$-modules and $\iCh$ the underlying $\infty$-category.

\item Given a topological operad $\cO$, we denote by $\alg_{\cO}(\cC)$ the category of $\cO$-algebras in a symmetric monoidal category $\cC$ and by $\ialg_{\cO}(\cC)$ the $\infty$-category of $\cO$-algebras in a symmetric monoidal $\infty$-category $\cC$. If $\cO$ is a dg operad, the category of $\cO$-algebras in complexes is simply denoted by $\alg_{\cO}$.

\item All operads are non-unital unless specified otherwise. We denote by $\cO^{\un}$ the operad of unital $\cO$-algebras.

\item All non-counital coalgebras are conilpotent.
\end{itemize}

\subsection*{Acknowledgements}

The author would like to thank D. Calaque, B. Hennion and N. Rozenblyum for useful conversations and V. Hinich for pointing out a mistake in the previous version. The main theorem (Theorem \ref{thm:Poissonadditivity}) was first announced by Nick Rozenblyum in 2013, who has given several talks about it. Unfortunately, his proof is not yet publicly available. This article presents another proof of the result which is different from Rozenblyum's.

\section{Operads}

\subsection{Relative categories}

In the paper we will extensively use relations between relative categories and $\infty$-categories, so let us recall the necessary facts.

\begin{defn}
A \defterm{relative category} $(\cC, W)$ consists of a category $\cC$ and a subcategory $W\subset \cC$ which has the same objects as $\cC$ and contains all isomorphisms in $\cC$.
\end{defn}

We will call morphisms belonging to $W$ \defterm{weak equivalences}. A functor of relative categories $(\cC, W_\cC)\rightarrow (\cD, W_\cD)$ is a functor that preserves weak equivalences.

Recall that given a category $\cC$ its nerve $\N(\cC)$ is an $\infty$-category. Similarly, if $\cC$ is a relative category, the nerve $\N(\cC)$ is an $\infty$-category equipped with a system of morphisms $W$ and we intoduce the notation
\[\cC[W^{-1}] = \N(\cC)[W^{-1}],\]
where the localization functor on the right is defined in \cite[Proposition 4.1.7.2]{HA}. In particular, $\cC[W^{-1}]$ is an $\infty$-category which we call the \defterm{underlying $\infty$-category} of the relative category $(\cC, W)$.

We will also need a construction of symmetric monoidal $\infty$-categories from ordinary symmetric monoidal categories. We say that $\cC$ is a \defterm{relative symmetric monoidal category} if the functor $x\otimes -\colon \cC\rightarrow \cC$ preserves weak equivalences for every object $x\in\cC$.

\begin{prop}
Let $\cC$ be a relative symmetric monoidal category. Then the localization $\cC[W^{-1}]$ admits a natural structure of a symmetric monoidal $\infty$-category.

If $F\colon \cC_1\rightarrow \cC_2$ is a (lax) symmetric monoidal functor of relative symmetric monoidal categories, then its localization induces a (lax) symmetric monoidal functor of $\infty$-categories
\[F\colon \cC_1[W^{-1}]\to\cC_2[W^{-1}].\]
\label{prop:smlocalization}
\end{prop}
\begin{proof}
Given a symmetric monoidal category $\cC$ we can construct the symmetric monoidal $\infty$-category $\cC^{\otimes}$ as in \cite[Construction 2.0.0.1]{HA}. The class of weak equivalences $W$ defines a system in the underlying $\infty$-category of $\cC^{\otimes}$ which is compatible with the tensor product and hence by \cite[Proposition 4.1.7.4]{HA} we can construct a symmetric monoidal $\infty$-category $(\cC')^{\otimes}[W^{-1}]$ whose underlying $\infty$-category is equivalent to $\cC[W^{-1}]$.

A (lax) symmetric monoidal functor $\cC_1\rightarrow \cC_2$ gives rise to a (lax) symmetric monoidal functor $\cC_1^{\otimes}\rightarrow \cC_2^{\otimes}$ of $\infty$-categories. Consider the composite
\[\cC_1^{\otimes}\to \cC_2^{\otimes}\to \cC_2^{\otimes}[W^{-1}].\]
Since $F\colon \cC_1\rightarrow \cC_2$ preserves weak equivalences, by the universal property of the localization we obtain a (lax) symmetric monoidal functor
\[F\colon \cC_1^{\otimes}[W^{-1}]\to \cC_2^{\otimes}[W^{-1}].\]
\end{proof}

For instance, let $\Ch$ be the symmetric monoidal category of chain complexes of $k$-vector spaces. Let $\qis\subset \Ch$ be the class of quasi-isomorphisms. Since $M\otimes -$ preserves quasi-isomorphisms for any $M\in\Ch$, we obtain a natural symmetric monoidal structure on the $\infty$-category
\[\iCh=\Ch[\qis^{-1}]\] of chain complexes.

We will repeatedly use the following method to prove that a functor between $\infty$-categories is an equivalence. Consider a commutative diagram of $\infty$-categories
\[
\xymatrix{
\cC_1 \ar_{G_1}[dr] \ar^{F}[rr] && \cC_2 \ar^{G_2}[dl]\\
& \cD &
}
\]
Assuming $G_1$ and $G_2$ have left adjoints $G_1^L$ and $G_2^L$ respectively, we obtain a natural transformation $G_2^L\rightarrow F G_1^L$ of functors $\cD\rightarrow \cC_2$. We say that the original diagram satisfies the \defterm{left Beck--Chevalley condition} if this natural transformation is an equivalence. The following is a corollary of the $\infty$-categorical version of the Barr--Beck theorem proved in \cite[Corollary 4.7.3.16]{HA}.

\begin{prop}
Suppose
\[
\xymatrix{
\cC_1 \ar_{G_1}[dr] \ar[rr] && \cC_2 \ar^{G_2}[dl]\\
& \cD &
}
\]
is a commutative diagram of $\infty$-categories such that
\begin{enumerate}
\item The functors $G_1$ and $G_2$ admit left adjoints.

\item The diagram satisfies the left Beck--Chevalley condition.

\item The $\infty$-categories $\cC_i$ admit and $G_i$ preserve geometric realizations of simplicial objects.

\item The functors $G_i$ are conservative.
\end{enumerate}

Then the functor $\cC_1\rightarrow \cC_2$ is an equivalence.
\label{prop:luriebarrbeck}
\end{prop}

\subsection{Operads}

Our definitions and notations for operads follows those of Loday and Vallette \cite{LV}. Unless specified otherwise, by an operad we mean an operad in chain complexes.

Given a symmetric sequence $V$, we define the shift $V[n]$ to be the symmetric sequence with
\[(V[n])(m) = V(m)[n].\]
Let $\sgn_m$ be the one-dimensional sign representation of $S_m$. We will also use the notation $V\{n\}$ to denote the symmetric sequence with
\[(V\{n\})(m) = V(m)\otimes \sgn_m^{\otimes n}[n(m-1)].\]
If $V$ is an operad or a cooperad, so is $V\{n\}$.

Let $\bu$ be the trivial operad. Recall that an \defterm{augmentation} on an operad $\cO$ is a morphism of operads
\[\cO\to \bu.\]
In particular, one obtains a splitting of symmetric sequences
\[\cO\cong \overline{\cO}\oplus \bu.\]
Similarly, one has a notion of a \defterm{coaugmentation} on a cooperad $\cC$.

Given an augmented operad $\cO$, its bar construction $\B\cO$ is defined to be the cofree cooperad on $\overline{\cO}[1]$ equipped with the bar differential which consists of two terms: one coming from the differential on $\cO$ and one coming from the product on $\cO$. Similarly, given a coaugmented cooperad $\cC$ we have its cobar construction $\Omega\cC$. We refer to \cite[Section 6.5]{LV} for details. In particular, one has a quasi-isomorphism of operads
\[\Omega\B\cO\stackrel{\sim}\to \cO.\]

Most of the operads of interest that control non-unital algebras satisfy $\cO(0)=0$ and $\cO(1)\cong k$ and hence possess a unique augmentation. However, operads controlling unital algebras tend not to have an augmentation, so, following Hirsh and Mill\`{e}s \cite{HM}, we relax the condition a bit.

\begin{defn}
A \defterm{semi-augmentation} on an operad $\cO$ is a morphism of the underlying graded symmetric sequences $\epsilon\colon\cO\rightarrow \bu$ which is not necessarily compatible with the differential and the product such that the composite
\[\bu\to \cO\stackrel{\epsilon}\to \bu\]
is the identity.
\end{defn}

Given a semi-augmented operad $\cO$, one can still consider the bar construction $\B\cO$, but the corresponding differential no longer squares to zero. Instead, we obtain a curved cooperad equipped with a curving $\theta\colon \cC(1)\rightarrow k[2]$ (see \cite[Definition 3.2.1]{HM} for a complete definition).

We refer to \cite[Section 3.3]{HM} for explicit formulas for the differential and the curving on the bar construction of a semi-augmented operad. Moreover, it is also shown there that the cobar construction $\Omega\cC$ on a coaugmented curved cooperad $\cC$ is a dg operad equipped with a natural semi-augmentation.

Finally, we refer to \cite[Section 7]{LV} for Koszul duality for augmented operads and to \cite[Section 4]{HM} for Koszul duality for semi-augmented operads. The important point that we will use in the paper is that the Koszul dual cooperad $\cC$ of $\cO$ is naturally equipped with a quasi-isomorphism
\[\Omega\cC\stackrel{\sim}\to \cO\]
which gives a semi-free resolution of the operad $\cO$. Such a quasi-isomorphism is equivalently given by a degree 1 (curved) Koszul twisting morphism $\cC\rightarrow \cO$.

\subsection{Operadic algebras}

Given an operad $\cO$ we denote by $\alg_{\cO}$ the category of $\cO$-algebras in chain complexes. Similarly, for a cooperad $\cC$ we denote by $\coalg_{\cC}$ the category of conilpotent $\cC$-coalgebras. To simplify the notation, we let
\[\alg = \alg_{\Ass^{\un}}\]
be the category of unital associative algebras.

If $\cC$ is a curved cooperad, we denote by $\coalg_{\cC}$ the category of curved conilpotent $\cC$-coalgebras (see \cite[Definition 5.2.1]{HM}) which are cofibrant. Note that morphisms strictly preserve the differential.

\begin{remark}
Positselski in \cite[Section 9]{Pos} considers a closely related category of curved coassociative coalgebras $k-\mathrm{coalg}_{\mathrm{cdg}}$ whose morphisms do not strictly preserve the differential.
\end{remark}

If $A$ is an $\cO$-algebra, then $A[-n]$ is an $\cO\{n\}$-algebra; similarly, if $C$ is a $\cC$-coalgebra, then $C[-n]$ is a $\cC\{n\}$-coalgebra.

Now consider a (curved) cooperad $\cC$ equipped with a (curved) Koszul twisting morphism $\cC\rightarrow \cO$. Given an $\cO$-algebra $A$ we define its bar construction to be
\[\B(A) = \cC(A) = \bigoplus_{n=0}^\infty (\cC(n)\otimes A^{\otimes n})^{S_n}\]
equipped with the bar differential (see \cite[Section 5.2.3]{HM}). Given a curved $\cC$-coalgebra $C$ we define its cobar construction to be
\[\Omega(C) = \cO(C) = \bigoplus_{n=0}^\infty (\cO(n)\otimes C^{\otimes n})_{S_n}\]
equipped with the cobar differential (see \cite[Section 5.2.5]{HM}). Note that the cobar differential squares to zero. In particular, we get a bar-cobar adjunction
\[\adj{\Omega\colon \coalg_{\cC}}{\alg_{\cO}\colon \B}\]
such that for any $\cO$-algebra $A$ the natural projection
\[\Omega\B A\to A\]
is a quasi-isomorphism.

Let us denote by $\qis\subset \alg_{\cO}$ the class of morphisms of $\cO$-algebras which are quasi-isomorphisms of the underlying complexes. Let us also denote by $\kos\subset \coalg_{\cC}$ the class of morphisms of (curved) $\cC$-coalgebras which become quasi-isomorphisms after applying the cobar functor $\Omega$. The class of weak equivalences $\kos$ is independent of the choice of the operad $\cO$ as shown by the following statement. Let us denote by $\Omega_{\cO}\colon \coalg_{\cC}\rightarrow \alg_{\cO}$ the cobar construction associated to the operad $\cO$.

\begin{prop}
Suppose $\cC\rightarrow \cO_1$ is a (curved) Koszul twisting morphism and $\cO_1\rightarrow \cO_2$ a quasi-isomorphism of operads. Consider a morphism of (curved) $\cC$-coalgebras $C_1\rightarrow C_2$. Then
\[\Omega_{\cO_1}(C_1)\to \Omega_{\cO_1}(C_2)\]
is a quasi-isomorphism iff
\[\Omega_{\cO_2}(C_1)\to \Omega_{\cO_2}(C_2)\]
is a quasi-isomorphism.
\label{prop:weindependent}
\end{prop}
\begin{proof}
Consider a commutative diagram
\[
\xymatrix{
\Omega_{\cO_1}(C_1) \ar[r] \ar[d] & \Omega_{\cO_1}(C_2) \ar[d] \\
\Omega_{\cO_2}(C_1) \ar[r] & \Omega_{\cO_2}(C_2).
}
\]

The morphisms
\[\Omega_{\cO_1}(C_i)\to \Omega_{\cO_2}(C_i)\]
are filtered quasi-isomorphisms where the filtration is defined as in \cite[Proposition 2.3]{Val}. The filtration is also complete and bounded below and hence the morphisms $\Omega_{\cO_1}(C_i)\rightarrow \Omega_{\cO_2}(C_i)$ are quasi-isomorphisms.

Therefore, the top morphism is a quasi-isomorphism iff the bottom morphism is a quasi-isomorphism.
\end{proof}

\begin{prop}
Suppose $\cC\rightarrow \cO$ is a (curved) Koszul twisting morphism. Then the adjunction
\[\adj{\Omega\colon \coalg_{\cC}}{\alg_{\cO}\colon \B}\]
descends to an adjoint equivalence of $\infty$-categories
\[\adj{\Omega\colon \coalg_{\cC}[\kos^{-1}]}{\alg_{\cO}[\qis^{-1}]\colon \B}.\]
\label{prop:koszulduality}
\end{prop}
\begin{proof}
First of all, we have to show that $\B$ and $\Omega$ preserve weak equivalences. Indeed, by definition $\Omega$ creates weak equivalences. Now suppose $A_1\rightarrow A_2$ is a quasi-isomorphism and consider the morphism $\B A_1\rightarrow \B A_2$. We have a commutative diagram
\[
\xymatrix{
\Omega\B A_1 \ar[d] \ar[r] & \Omega\B A_2 \ar[d] \\
A_1 \ar[r] & A_2
}
\]
where the two vertical morphisms are quasi-isomorphisms by \cite[Proposition 5.2.8]{HM} and the bottom morphism is a quasi-isomorphism by assumption. Therefore, $\Omega \B A_1\rightarrow \Omega\B A_2$ is also a quasi-isomorphism and hence by definition $\B A_1\rightarrow \B A_2$ is a weak equivalence.

Therefore, we get an adjunction $\Omega\dashv \B$ of the underlying $\infty$-categories. To show that it is an adjoint equivalence we have to show that the unit and the counit of the adjunction are weak equivalences. Indeed, again by \cite[Proposition 5.2.8]{HM} the counit of the adjunction is a weak equivalence. Next, suppose $C$ is a (curved) $\cC$-coalgebra and consider the unit of the adjunction $C\rightarrow \B\Omega C$.
To show that it is a weak equivalence, consider the morphisms
\[\Omega C\to \Omega\B\Omega C\to \Omega C.\]
By construction the composite morphism is the identity; the second morphism is the counit of the adjunction hence is a quasi-isomorphism. Therefore, the first morphism is a quasi-isomorphism and hence $C\rightarrow \B\Omega C$ is a weak equivalence.
\end{proof}

We introduce the notation
\[\ialg_{\cO} = \alg_{\cO}[\qis^{-1}]\]
for the $\infty$-category of $\cO$-algebras and by the previous proposition it can also be modeled by the relative category $(\coalg_{\cC}, \kos)$.

Suppose $\cO_1\rightarrow \cO_2$ is a morphism of operads which is a quasi-isomorphism in each arity. The forgetful functor
\[\alg_{\cO_2}\to \alg_{\cO_1}\]
automatically preserves quasi-isomorphisms and hence induces a functor on the level of $\infty$-categories.

\begin{prop}
Let $\cO_1\rightarrow \cO_2$ be a quasi-isomorphism of operads. Then the forgetful functor
\[\alg_{\cO_2}\to \alg_{\cO_1}\]
induces an equivalence of $\infty$-categories
\[\ialg_{\cO_2}\to \ialg_{\cO_1}\]
\label{prop:operadforgetful}
\end{prop}
\begin{proof}
Indeed, by \cite[Theorem 4.1]{BM} and \cite[Theorem 4.7.4]{Hin} the induction/restriction functors provide a Quillen equivalence between $\alg_{\cO_1}$ and $\alg_{\cO_2}$ and hence induce an equivalence of the underlying $\infty$-categories.
\end{proof}

Finally, recall that the forgetful functor $\alg_{\cO}\rightarrow \Ch$ creates sifted colimits since the category $\alg_{\cO}$ can be written as the category of algebras over a monad which preserves sifted colimits. The same statement is true on the level of $\infty$-categories.

\begin{prop}
Let $I$ be a set and $\cO$ an $I$-colored dg operad. The forgetful functor
\[\ialg_{\cO}\to\Fun(I, \iCh)\]
creates sifted colimits.
\label{prop:forgetsifted}
\end{prop}
\begin{proof}
By \cite[Proposition 1.3.4.24]{HA} we just need to show that the forgetful functor $\alg_{\cO}\rightarrow \Fun(I, \Ch)$ creates homotopy sifted colimits which follows by \cite[Proposition 7.8]{PS}.
\end{proof}

\subsection{Examples}
\label{sect:examples}

Let us show how the bar-cobar duality works for unital algebras over the associative and Poisson operads to compare it to the classical bar-cobar duality.

\begin{figure}
\begin{tikzpicture}
\node[w] (v) at (0, 0) {$u$};
\node (root) at (0, -1) {};
\draw (v) edge (root);
\end{tikzpicture}
\qquad
\begin{tikzpicture}
\node[w] (v) at (0, 0) {$m$};
\node (root) at (0, -1) {};
\node (v1) at (-0.5, 1) {};
\node (v2) at (0.5, 1) {};
\draw (v) edge (root);
\draw (v) edge (v1);
\draw (v) edge (v2);
\end{tikzpicture}
\caption{Generating operations of $\Ass^{\un}$.}
\label{fig:assgenerating}
\end{figure}
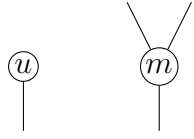

We begin with the case of the associative operad considered in \cite[Section 6]{HM}. Let $\cO=\Ass^{\un}$ be the operad governing unital associative algebras. The operad $\Ass^{\un}$ is quadratic and is generated by the symmetric sequence $V$ with $V(0) = k$ and $V(2) = k[S_2]$ whose elements are shown on Figure \ref{fig:assgenerating}.

\begin{figure}
\begin{tikzpicture}
\node[w] (v0l) at (-1, 0) {$m$};
\node (rootl) at (-1, -1) {};
\node[w] (v1l) at (-1.5, 0.7) {$m$};
\node (v11l) at (-1.8, 1.4) {};
\node (v12l) at (-1.2, 1.4) {};
\node (v2l) at (-0.6, 1.4) {};
\draw (v0l) edge (rootl);
\draw (v0l) edge (v1l);
\draw (v1l) edge (v11l);
\draw (v1l) edge (v12l);
\draw (v0l) edge (v2l);

\node at (0, 0) {$=$};

\node[w] (v0r) at (1, 0) {$m$};
\node (rootr) at (1, -1) {};
\node[w] (v2r) at (1.5, 0.7) {$m$};
\node (v1r) at (0.6, 1.4) {};
\node (v21r) at (1.2, 1.4) {};
\node (v22r) at (1.8, 1.4) {};
\draw (v0r) edge (rootr);
\draw (v0r) edge (v1r);
\draw (v0r) edge (v2r);
\draw (v2r) edge (v21r);
\draw (v2r) edge (v22r);
\end{tikzpicture}
\qquad
\begin{tikzpicture}
\node[w] (vl) at (-1, 0) {$m$};
\node (rootl) at (-1, -1) {};
\node[w] (v1l) at (-1.5, 1) {$u$};
\node (v2l) at (-0.5, 1) {};
\draw (vl) edge (rootl);
\draw (vl) edge (v1l);
\draw (vl) edge (v2l);

\node at (0, 0) {$=$};

\node (vr) at (1, 1) {};
\node (rootr) at (1, -1) {};
\draw (vr) edge (rootr);
\end{tikzpicture}
\qquad
\begin{tikzpicture}
\node[w] (vl) at (-1, 0) {$m$};
\node (rootl) at (-1, -1) {};
\node (v1l) at (-1.5, 1) {};
\node[w] (v2l) at (-0.5, 1) {$u$};
\draw (vl) edge (rootl);
\draw (vl) edge (v1l);
\draw (vl) edge (v2l);

\node at (0, 0) {$=$};

\node (vr) at (1, 1) {};
\node (rootr) at (1, -1) {};
\draw (vr) edge (rootr);
\end{tikzpicture}
\caption{Relations in $\Ass^{\un}$.}
\label{fig:assrelations}
\end{figure}

The relations in $\Ass^{\un}$ have the form shown on Figure \ref{fig:assrelations}. This gives an inhomogeneous quadratic-linear-constant presentation of the operad $\Ass^{\un}$. Given such an operad $\cO$, we denote by $q\cO$ the operad with the same generators and where we only keep the quadratic part of the relations. Recall that the underlying graded cooperad of the Koszul dual is defined from the quadratic part of the relations, the differential uses the linear part and the curving comes from the constant part. In the relations we have there are no linear terms, so the Koszul dual cooperad coincides with the Koszul dual cooperad of the quadratic operad $q\Ass^{\un}$ equipped with a curving. From the relations we see that
\[q\Ass^{\un} \cong \Ass\oplus \E_0,\]
where $\E_0$ is the operad governing complexes together with a distinguished vector and where $\oplus$ refers to the product in the category of operads.

\begin{figure}
\begin{tikzpicture}
\node[w] (v) at (0, 0) {$\Delta$};
\node (root) at (0, 1) {};
\node[w] (v1) at (-0.5, -1) {$u$};
\node (v2) at (0.5, -1) {};
\draw (v) edge (root);
\draw (v) edge (v1);
\draw (v) edge (v2);
\end{tikzpicture}
\qquad
\begin{tikzpicture}
\node[w] (v) at (0, 0) {$\Delta$};
\node (root) at (0, 1) {};
\node (v1) at (-0.5, -1) {};
\node[w] (v2) at (0.5, -1) {$u$};
\draw (v) edge (root);
\draw (v) edge (v1);
\draw (v) edge (v2);
\end{tikzpicture}
\caption{Curving on $(\Ass^{\un})^{\ish}$.}
\label{fig:asscurving}
\end{figure}

By \cite[Proposition 6.1.4]{HM} the Koszul dual of $q\Ass^{\un}$ is given by
\[(q\Ass^{\un})^\ish = \Ass^{\ish}\oplus \coE_0\{-1\},\]
where $\coE_0 \cong \E_0$ is the cooperad of complexes together with a functional and $\oplus$ now denotes the product in the category of conilpotent cooperads, i.e. the conilpotent cooperad cogenerated by $\Ass^{\ish}$ and $\coE_0$. We define
\[(\Ass^{\un})^\ish = (q\Ass^{\un})^\ish\]
as graded cooperads. The degree $-2$ part of $(\Ass^{\un})^\ish(1)$ is the two-dimensional vector space spanned by trees shown in Figure \ref{fig:asscurving} and we define the curving $\theta\colon (\Ass^{\un})^\ish(1)\rightarrow k[2]$ to take value $-1$ on both of these. Let us denote
\[\coAss^\theta\{1\} = (\Ass^{\un})^\ish.\]
The cooperad $\coAss^\theta$ governs coassociative coalgebras $C$ together with a coderivation $\d\colon C\rightarrow C$ of degree $1$ and a curving $\theta\colon C\rightarrow k[2]$ satisfying the equations
\begin{align*}
\d^2 x &= \theta(x_{(1)}) x_{(2)} - x_{(1)} \theta(x_{(2)}) \\
\theta(\d x) &= 0
\end{align*}
where
\[\Delta(x) = x_{(1)}\otimes x_{(2)}.\]

Given a unital associative dg algebra $A$, its bar complex is
\[\B(A) = \overline{\T}_\bullet(A[1]\oplus k[2])\]
equipped with the following differential. Let us denote elements of the bar complex by $[x_1|...|x_n]$ with elements of $k[2]$ denoted by $\ast$. Then
\begin{align*}
\d[x_1|...|x_n] &= \sum_{i=1}^n(-1)^{\sum_{q=1}^{i-1}(|x_q|+1)}[x_1|...|\d x_i|...|x_n] \\
&+\sum_{i=1}^{n-1}(-1)^{\sum {q=1}^i(|x_q|+1)}[x_1|...|x_ix_{i+1}|...|x_n]
\end{align*}
with $\d\ast = 1\in A$ and $\ast\cdot x = x\cdot \ast = 0$. The curving is given by $\theta([\ast]) = 1$.

\begin{remark}
Given a unital associative dg algebra $A$ equipped with a semi-augmentation, Positselski in \cite[Section 6.1]{Pos} considers the bar construction to be $\T_\bullet(\overline{A}[1])$ equipped with a natural curving and the standard bar differential.
\end{remark}

Similarly, given a curved coalgebra $C$, its cobar complex is
\[\Omega(C) = \T^\bullet(C[-1])\]
whose elements we denote by $[x^1|...|x^n]$ for $x^i\in C$ such that the differential is
\begin{align*}
\d[x^1|...|x^n] &= \sum_{i=1}^n (-1)^{\sum_{q=1}^{i-1}(|x^q|+1)}[x^1|...|\d x^i|...|x^n] \\
&+ \sum_{i=1}^n (-1)^{\sum_{q=1}^{i-1}(|x^q|+1)}[x^1|...|\theta(x^i)x^{i+1}|...|x^n] \\
&+ \sum_{i=1}^n(-1)^{\sum_{q=1}^{i-1}(|x^q|+1)+|x^i_{(1)}|+1}[x^1|...|x^i_{(1)}|x^i_{(2)}|...|x^n]
\end{align*}

Let us similarly work out the Koszul dual of the operad of unital $\P_n$-algebras. Recall that a $\P_n$-algebra is a dg Poisson algebra whose Poisson bracket has degree $1-n$. We denote by $\cO=\P_n^{\un}$ the operad controlling such algebras. It is generated by the symmetric sequence $V$ with $V(0) = k\cdot u$ and $V(2) = k\cdot m \oplus \sgn_2^{\otimes n}\cdot \{\}$ with the following relations:
\begin{align*}
a(bc) &= (ab)c \\
\{\{a, b\}, c\} &= (-1)^{n + |b||c|+1}\{\{a, c\}, b\} + (-1)^{|a|(|b|+|c|)+1}\{\{b, c\}, a\} \\
\{a, bc\} &= \{a, b\}c + (-1)^{|b||c|}\{a, c\}b \\
1\cdot a &= a \\
\{1, a\} &= 0.
\end{align*}

In particular, we see again that
\[q\P_n^{\un} \cong \P_n\oplus \E_0.\]

Note that the Koszul dual of the operad of non-unital $\P_n$-algebras is $\P_n^\ish\cong \coP_n\{n\}$. Therefore, by \cite[Proposition 6.1.4]{HM} the Koszul dual of $\P_n^{\un}$ is
\[(\P_n^{\un})^\ish\cong \coP_n\{n\}\oplus \coE_0\{-1\}\]
equipped with the curving $\theta\colon (\P_n^{\un})^\ish(1)\rightarrow k[2]$ which sends the tree

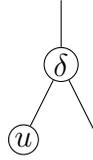
\begin{figure}[h]
\begin{tikzpicture}
\node[w] (v) at (0, 0) {$\delta$};
\node (root) at (0, 1) {};
\node[w] (v1) at (-0.5, -1) {$u$};
\node (v2) at (0.5, -1) {};
\draw (v) edge (root);
\draw (v) edge (v1);
\draw (v) edge (v2);
\end{tikzpicture}
\end{figure}

to $-1$, where $\delta$ is the cobracket in $\coP_n\{n\}$. We denote
\[\coP_n^\theta\{n\} = (\P_n^{\un})^\ish.\]

A curved coalgebra $C$ over the cooperad $\coP_n^\theta$ is given by the following data:
\begin{itemize}
\item A cocommutative comultiplication on $C$.

\item A cobracket $\delta\colon C\rightarrow C\otimes C[1-n]$ for which we use Sweedler's notation
\[\delta(x) = x^{\delta}_{(1)}\otimes x^{\delta}_{(2)}\]
which satisfies the coalgebraic version of the Jacobi identity.

\item A coderivation $\d\colon C\rightarrow C$ of degree $1$.

\item A curving $\theta\colon C\rightarrow k[1+n]$.
\end{itemize}
Together these satisfy the relations
\begin{align*}
\d^2 x &= \theta(x^{\delta}_{(1)})x^{\delta}_{(2)} \\
\theta(\d x) &= 0.
\end{align*}

\subsection{Hopf operads}

\label{sect:monoidal}

Recall that a \defterm{Hopf operad} is an operad in counital cocommutative coalgebras. Dually, a \defterm{Hopf cooperad} is a cooperad in unital commutative algebras. Alternatively, recall that the category of symmetric sequences has two tensor structures: the composition product which is merely monoidal and which we use to define operads and cooperads and the Hadamard product which is symmetric monoidal. The Hadamard tensor product defines a symmetric monoidal structure on the category of cooperads and one can define a Hopf cooperad to be a unital commutative algebra in the category of cooperads.

One can similarly define a notion of a curved Hopf cooperad to be a unital commutative algebra in the category of curved cooperads.

Given a Hopf operad $\cO$ and two $\cO$-algebras $A_1, A_2$ the tensor product of the underlying complexes is also an $\cO$-algebra using
\[\cO(n)\otimes (A\otimes B)^{\otimes n}\rightarrow \cO(n)\otimes A^{\otimes n}\otimes \cO(n)\otimes B^{\otimes n}\rightarrow A\otimes B.\]
Dually, one defines the tensor product of two (curved or dg) $\cC$-coalgebras for a Hopf cooperad $\cC$ to be the tensor product of the underlying chain complexes.

The operads $\Ass^{\un}$ and $\P_n^{\un}$ we are interested in are Hopf operads. For instance, for $\P_n^{\un}$ we have
\[\Delta(m) = m\otimes m\qquad \Delta(\{\}) = \{\}\otimes m + m\otimes \{\}.\]
By duality we get Hopf cooperad structures on $\coAss^{\cu}=(\Ass^{\un})^*$ and $\coP_n^{\cu}=(\P_n^{\un})^*$, the cooperads of counital coassociative coalgebras and counital $\P_n$-coalgebras respectively.

Note, however, that the curved cooperad $\coAss^\theta$ admits no Hopf structure. Indeed, the degree zero part of $\coAss^\theta(0)$ is trivial and hence one cannot define a unit. To remedy this problem, we introduce the following modification. Given an operad $\cO$ we denote by $\cO^{\un} = \cO\oplus k$ the symmetric sequence which coincides with $\cO$ in arities at least 1 and which is $\cO(0)\oplus k$ in arity zero. Similarly, we define the symmetric sequence $\cC^{\cu}=\cC\oplus k$ for a cooperad $\cC$.

\begin{defn}
A \defterm{Hopf unital structure} on an operad $\cO$ is the structure of a Hopf operad on $\cO^{\un}$ such that the natural inclusion $\cO\rightarrow \cO^{\un}$ is a morphism of operads and such that the counit on $\cO^{\un}(0)=\cO(0)\oplus k$ is given by the projection on the second factor.
\label{def:hopfunital}
\end{defn}

\begin{defn}
A \defterm{Hopf counital structure} on a (curved) cooperad $\cC$ is the structure of a (curved) Hopf cooperad on $\cC^{\cu}$ such that the natural projection $\cC^{\cu}\rightarrow \cC$ is a morphism of (curved) cooperads and such that the unit on $\cC^{\cu}(0) = \cC(0)\oplus k$ is given by inclusion into the second factor.
\label{def:hopfcounital}
\end{defn}

For instance, $\Ass$ and $\P_n$ have Hopf unital structures given by the Hopf operads $\Ass^{\un}$ and $\P_n^{\un}$. Similarly, the curved cooperads $\coAss^\theta$ and $\coP_n^\theta$ have Hopf counital structures given by the cooperads $\coAss^{\theta, \cu}$ and $\coP_n^{\theta, \cu}$ respectively, where, for instance, $\coAss^{\theta, \cu}$ governs curved counital coassociative coalgebras.

Given a Hopf operad $\cO^{\un}$, the counits assemble to give a map of operads $\cO^{\un}\rightarrow \Comm^{\un}$. We can equivalently define a Hopf unital structure on an operad $\cO$ to be a Hopf operad $\cO^{\un}$ together with an isomorphism of operads \[\cO\cong \cO^{\un}\times_{\Comm^{\un}} \Comm.\]
Using the morphism $\cO^{\un}\rightarrow \Comm^{\un}$ we see that the unital commutative algebra $k$ admits a natural $\cO^{\un}$-algebra structure.

\begin{defn}
Let $\cO^{\un}$ be a Hopf operad. An \defterm{augmented $\cO^{\un}$-algebra} is an $\cO^{\un}$-algebra $A$ together with a morphism of $\cO^{\un}$-algebras $A\rightarrow k$.
\end{defn}
Dually, for a Hopf cooperad $\cC^{\cu}$ we define the notion of a coaugmented $\cC^{\cu}$-coalgebra. We denote by $\alg_{\cO^{\un}}^{\aug}$ the category of augmented $\cO^{\un}$-algebras and by $\coalg_{\cC^{\cu}}^{\coaug}$ the category of coaugmented $\cC^{\cu}$-coalgebras. Note that a coaugmented $\cC^{\cu}$-coalgebra $C\rightarrow k$ is automatically assumed to be conilpotent in the sense that the non-counital coalgebra $\overline{C}$ is conilpotent.

\begin{lm}
Let $\cO$ be an operad with a Hopf unital structure. Then we have an equivalence of categories
\[\alg_{\cO}\cong \alg_{\cO^{\un}}^{\aug}.\]

Dually, if $\cC$ is a dg cooperad with a Hopf counital structure, then we have an equivalence of categories
\[\coalg_{\cC}\cong \coalg_{\cC^{\cu}}^{\coaug}.\]
\label{lm:addunit}
\end{lm}
\begin{proof}
Given an $\cO^{\un}$-algebra $A$, we have the unit
\[k\to A\]
coming from the inclusion $k\rightarrow \cO^{\un}(0)\cong \cO(0)\oplus k$ into the second factor. The coaugmentation $A\rightarrow k$ splits the unit and hence one has an isomorphism of complexes
\[A\cong \overline{A}\oplus k.\]

The $\cO^{\un}$-algebra structure on $A$ gives rise to the operations
\[\cO(0)\to A,\qquad \cO(n)\otimes \overline{A}^{\otimes m}\to A\]
where $m\leq n$.

The morphism $\cO(0)\rightarrow A\rightarrow k$ is zero by the definition of the counit on $\cO^{un}(0)$. The morphisms $\cO(n)\otimes \overline{A}^{\otimes m}\rightarrow A$ of $m < n$ are uniquely determined from the ones for $m=n$. But since $A\rightarrow k$ is a morphism of $\cO^{\un}$-algebras, the diagram
\[
\xymatrix{
\cO(n)\otimes \overline{A}^{\otimes n} \ar^{0}[d] \ar[r] & A \ar[d] \\
\cO(n) \ar[r] & k
}
\]
is commutative which implies that the composite
\[\cO(n)\otimes \overline{A}^{\otimes m}\rightarrow A\rightarrow k\]
factors through $\overline{A}$. Therefore, the augmented $\cO^{\un}$-algebra structure on $A$ is uniquely determined by the $\cO$-algebra structure on $\overline{A}$.

The statement for coalgebras is proved similarly.
\end{proof}

The same construction works for a curved cooperad $\cC$. Note that since $\cC^{\cu}$ is a Hopf cooperad, the category $\coalg_{\cC}$ inherits a natural symmetric monoidal structure. Explicitly, given two $\cC$-coalgebras $C_1, C_2$, the underlying graded vector space of their tensor product is defined to be $C_1\otimes C_2\oplus C_1\oplus C_2$.

\begin{remark}
The cooperads $\coAss$ and $\coP_n$ are already Hopf cooperads and this gives a \emph{different} symmetric monoidal structure on the category $\coalg_{\coP_n}$ which we will not consider in the paper.
\end{remark}

We will need a certain compatibility between the Hopf structure on $\cC^{\cu}$ and weak equivalences.
\begin{defn}
A Hopf unital structure on a (curved) cooperad $\cC$ is \defterm{admissible} if the tensor product functor
\[C\otimes -\colon \coalg_{\cC^{\cu}}^{\coaug}\to \coalg_{\cC^{\cu}}^{\coaug}\]
preserves weak equivalences for any $C\in \coalg_{\cC^{\cu}}^{\coaug}$.
\end{defn}

We are now going to show that some standard examples of Hopf cooperads are admissible.
\begin{prop}
One has the following quasi-isomorphisms of $\cO$-algebras for any pair $C_1, C_2$ of (curved) conilpotent $\cC$-coalgebras:
\begin{enumerate}
\item For $\cO=\Ass\{-1\}$ and $\cC=\coAss$
\[\Omega(C_1\oplus C_2\oplus C_1\otimes C_2)\longrightarrow \Omega(C_1)\oplus \Omega(C_2)\oplus \Omega(C_1)\otimes \Omega(C_2)[-1].\]

\item For $\cO=\Lie\{-1\}$ and $\cC=\coComm$
\[\Omega(C_1\oplus C_2\oplus C_1\otimes C_2)\longrightarrow \Omega(C_1)\oplus \Omega(C_2).\]

\item For $\cO=\P_n\{-n\}$ and $\cC=\coP_n$
\[\Omega(C_1\oplus C_2\oplus C_1\otimes C_2)\longrightarrow \Omega(C_1)\oplus \Omega(C_2)\oplus \Omega(C_1)\otimes \Omega(C_2)[-n].\]

\item For $\cO=\Ass^{\un}\{-1\}$ and $\cC=\coAss^\theta$
\[\Omega(C_1\oplus C_2\oplus C_1\otimes C_2)\longrightarrow \Omega(C_1)\otimes \Omega(C_2)[-1].\]

\item For $\cO=\P_n^{\un}\{-n\}$ and $\cC=\coP_n^\theta$
\[\Omega(C_1\oplus C_2\oplus C_1\otimes C_2)\longrightarrow \Omega(C_1)\otimes \Omega(C_2)[-n].\]
\end{enumerate}
\label{prop:barlaxmonoidal}
\end{prop}
\begin{proof}
For simplicity in cases (1) and (3) we add units and augmentations using Lemma \ref{lm:addunit}.

In the case $\cC=\coAss$ a morphism of associative algebras
\[p\colon A=\T^\bullet(C_1[-1]\oplus C_2[-1]\oplus C_1\otimes C_2[-1])\longrightarrow \T^\bullet(C_1[-1])\otimes \T^\bullet(C_2[-1])\]
is uniquely specified on the generators and we define it to be zero on $C_1\otimes C_2[-1]$ and the obvious inclusions on the first two summands. It is clear that $p$ is compatible with the cobar differentials.

Elements of $A$ are given by words
\[[x_1|...|y_k|...|(x_i,y_i)|...],\]
where $x_n\in C_1$, $y_n\in C_2$ and $(x_n,y_n)\in C_1\otimes C_2$. The coproduct on $C_1\otimes C_2$ is given by
\begin{align*}
\Delta(x, y) &= x\otimes y + (-1)^{|x||y|} y\otimes x + (-1)^{|x_{(2)}||y_{(1)}|}(x_{(1)}, y_{(1)})\otimes (x_{(2)}, y_{(2)}) + (x, y_{(1)})\otimes y_{(2)} \\
&+ (-1)^{|x||y_{(1)}|} y_{(1)}\otimes (x, y_{(2)}) + x_{(1)}\otimes (x_{(2)}, y) + (-1)^{|y||x_{(2)}|}(x_{(1)}, y)\otimes x_{(2)}.
\end{align*}

We define a splitting
\[i\colon \T^\bullet(C_1[-1])\otimes \T^\bullet(C_2[-1])\longrightarrow A\]
to be the multiplication, e.g.
\[i([x_1|x_2]\otimes [y_1|y_2|y_3]) = [x_1|x_2|y_1|y_2|y_3].\]
It is again clear that $i$ is compatible with the cobar differentials and that $p\circ i=\id$. Note, however, that $i$ is merely a morphism of chain complexes and is not compatible with the multiplication.

Let $\{F_n C_i\}_{n\geq 1}$ be the coradical filtrations on $C_i$ which satisfy
\[\Delta(F_n C_i)\subset \bigoplus_{a+b=n} F_a C_i\otimes F_b C_i.\]
We introduce a filtration on the coalgebra $C_1\oplus C_2\oplus C_1\otimes C_2$ by declaring that \[F_n C_1\otimes F_m C_2\subset F_{n+m-1}(C_1\otimes C_2).\] It induces a filtration on the algebra $A$ by declaring that $[x_1|...|x_n]$ with $x_i\in F_{n_i}$ lies in $F_{\sum n_i-n} A$ and similarly for $\T^\bullet(C_1[-1])\otimes \T^\bullet(C_2[-1])$. The filtrations are bounded below (by zero) and complete. Moreover, the morphism $p$ is clearly compatible with the filtrations, so it is enough to prove that $p$ is a quasi-isomorphism after passing to the associated graded algebras with respect to the filtration. This allows us to assume that the coproducts on $C_i$ are zero.

We are now going to construct a homotopy $\id\stackrel{h}\sim i\circ p$ such that
\begin{equation}
a - i(p(a)) = \d h(a) + h(\d a).
\label{eq:associativehomotopy}
\end{equation}

The homotopy $h$ on monomials $a\in A$ is constructed by the following algorithm akin to bubble sort:
\begin{itemize}
\item Suppose $a$ has no factors in $C_1\otimes C_2$. If $a$ has no elements in the wrong order (i.e. elements of $C_2$ followed by an element of $C_1$), then $h(a) = 0$. Otherwise, write $a = b\cdot [y|x]\cdot c$, where $c$ has no elements in the wrong order and define inductively
\begin{equation}
h(a) = (-1)^{|x||y|+|y|+1 + |b|}b\cdot [(x, y)]\cdot c + (-1)^{(|x|+1)(|y|+1)}h(b\cdot [x|y]\cdot c).
\label{eq:homotopystep1}
\end{equation}

\item If $a$ has factors in $C_1\otimes C_2$, we define it recursively by the formula
\begin{equation}
h(a) = (-1)^{|b| + |x| + |y| + 1}b\cdot [(x, y)]\cdot h(c),
\label{eq:homotopystep2}
\end{equation}
where $a = b\cdot [(x, y)]\cdot c$ such that $b$ has no factors in $C_1\otimes C_2$.
\end{itemize}

Clearly, the second step reduces the number of elements in $C_1\otimes C_2$, so in finite time we arrive at an expression without factors in $C_1\otimes C_2$. Given a monomial $a$ without factors in $C_1\otimes C_2$ we define the number of \defterm{inversions} to be the number of elements of $C_2$ left of an element of $C_1$. For instance, the expression $[y_1|x_1|x_2|y_2]$ with $x_i\in C_1$ and $y_i\in C_2$ has two inversions. It is immediate that the first step of the algorithm reduces the number of inversions by 1 and hence it also terminates in finite time.

Let us make a preliminary observation that the equation
\begin{align}
b\cdot [(x, y)] ip(c) &= (-1)^{|y|}b\cdot [x|y]\cdot h(c) + (-1)^{|x|(|y|+1)}b\cdot [y|x]\cdot h(c) \nonumber \\
&+ (-1)^{|b|+|x|+1} h(b\cdot [x|y]\cdot c) + (-1)^{|b|+|x||y|+|y|+1}h(b\cdot [y|x]\cdot c) \label{eq:homotopylemma}
\end{align}
holds if $b$ does not contain elements of $C_1\otimes C_2$. Indeed, if $c$ contains elements of $C_1\otimes C_2$, both sides are zero by equation \eqref{eq:homotopystep2}. Otherwise, it is enough to assume that $c$ is ordered. In that case
\[(-1)^{|b|+|x||y|+|y|+1}h(b\cdot [y|x]\cdot c) = b\cdot [(x, y)]\cdot c + (-1)^{|x|+|b|} h(b\cdot [x|y]\cdot c)\]
by equation \eqref{eq:homotopystep1} and hence equation \eqref{eq:homotopylemma} holds in this case. Let us now show that thus constructed homotopy $h$ satisfies equation \eqref{eq:associativehomotopy}.

An element $a$ with $0$ inversions is completely ordered and by definition $h$ annihilates it. $\d a$ also has $0$ inversions, so $h(\d a)=0$, but we also have $a = i(p(a))$. Next, suppose $a$ has no factors in $C_1\otimes C_2$. Suppose we have checked the formula \eqref{eq:associativehomotopy} for all monomials with at most $k$ inversions and consider a monomial $a=b\cdot [y|x]\cdot c$ with $k+1$ inversions. We have
\begin{align*}
\d h(a) &= (-1)^{|x||y|+|y|+1+|b|} \d b\cdot [(x, y)]\cdot c + (-1)^{|x||y|+|y|+1} b\cdot [(\d x, y)]\cdot c \\
&+ (-1)^{(|x|+1)(|y|+1)} b\cdot [(x, \d y)] \cdot c + (-1)^{|x||y|+|x|+|y|} b\cdot [x|y]\cdot c \\
&+ b\cdot [y|x]\cdot c + (-1)^{|x|(|y|+1)} b\cdot [(x, y)]\cdot \d c + (-1)^{(|x|+1)(|y|+1)} \d h(b\cdot [x|y]\cdot c),
\end{align*}
\begin{align*}
\d a &= \d b\cdot [y|x]\cdot c + (-1)^{|b|}b\cdot [\d y|x]\cdot c+(-1)^{|b|+|y|+1}b\cdot [y|\d x]\cdot c \\
&+ (-1)^{|b|+|x|+|y|} b\cdot [y|x]\cdot \d c
\end{align*}
and
\begin{align*}
h(\d a) &= (-1)^{|x||y|+|y|+|b|}\d b\cdot [(x, y)]\cdot c + (-1)^{(|x|+1)(|y|+1)} h(\d b\cdot [x|y]\cdot c) \\
&+(-1)^{|x||y|+|x|+|y|}b\cdot [(x,\d y)]\cdot c + (-1)^{(|x|+1)|y|+|b|}h(b\cdot [x|\d y]\cdot c) \\
&+(-1)^{(|x|+1)|y|}b\cdot [(\d x, y)]\cdot c + (-1)^{(|x|+1)(|y|+1)+|b|} h(b\cdot [\d x|y]\cdot c) \\
&+(-1)^{|x|(|y|+1)+1} b\cdot [(x, y)]\cdot \d c + (-1)^{|x||y|+|b|}h(b\cdot [x|y]\cdot \d c)
\end{align*}

We also have \[ip(b\cdot [x|y]c) = (-1)^{(|x|+1)(|y|+1)} ip(a)\]
and by inductive assumption
\[h(\d(b\cdot [x|y]\cdot c)) + \d h(b\cdot [x|y]\cdot c) = b\cdot [x|y]\cdot c - (-1)^{(|x|+1)(|y|+1)} ip(a).\]

Combining these equations we get
\begin{align*}
\d h(a) + h(\d a) &= (-1)^{|x||y|+|x|+|y|}b\cdot [x|y]\cdot c + a + (-1)^{(|x|+1)(|y|+1)}b\cdot [x|y]\cdot c - ip(a) \\
&= a - ip(a).
\end{align*}

So far we have proved equation \eqref{eq:associativehomotopy} for monomials with no factors in $C_1\otimes C_2$. Now suppose we know the formula holds for monomials with at most $k$ factors in $C_1\otimes C_2$ and consider a monomial $a$ with $k+1$ factors in $C_1\otimes C_2$. Then
\begin{align*}
\d h(a) &= (-1)^{|b|+|x|+|y|+1} \d b\cdot [(x, y)]\cdot h(c) + (-1)^{|x|+|y|+1}b\cdot [(\d x, y)]\cdot h(c) \\
&+ (-1)^{|y|+1}b\cdot [(x, \d y)]\cdot h(c) + b\cdot [(x, y)]\cdot \d h(c) + (-1)^{|y|} b\cdot [x|y]\cdot h(c) \\
&+ (-1)^{|x|(|y|+1)}b\cdot [y|x]\cdot h(c),
\end{align*}
\begin{align*}
\d a &= \d b\cdot [(x, y)]\cdot c + (-1)^{|b|} b\cdot [(\d x, y)]\cdot c + (-1)^{|b|+|x|} b\cdot [(x, \d y)]\cdot c \\
&+ (-1)^{|b|+|x|+1}b\cdot [x|y]\cdot c + (-1)^{|b|+|x||y|+|y|+1}b\cdot [y|x]\cdot c \\
&+ (-1)^{|b|+|x|+|y|+1}b\cdot [(x,  y)]\cdot \d c
\end{align*}
and
\begin{align*}
h(\d a) &= (-1)^{|b|+|x|+|y|}\d b\cdot [(x, y)]\cdot h(c) + (-1)^{|x|+|y|}b\cdot [(\d x, y)]\cdot h(c) + (-1)^{|y|}b\cdot [(x, \d y)]\cdot h(c) \\
&+ (-1)^{|b|+|x|+1}h(b\cdot [x|y]\cdot c) + (-1)^{|b|+|x||y|+|y|+1}h(b\cdot [y|x]\cdot c) + b\cdot [(x, y)]\cdot h(\d c).
\end{align*}

Therefore,
\[h(\d a) + \d h(a) = a\]
by using equation \eqref{eq:homotopylemma}. This finishes the proof for $\cC=\coAss$.

Observe now that the morphism $p$ is compatible with the shuffle coproducts on both sides by looking at the generators. If $C_1$ and $C_2$ are both cocommutative, then the shuffle coproduct is compatible with the differentials. Therefore, passing to the primitives we obtain the statement for $\cC=\coComm$.

In the case $\cC=\coP_n$ we can assume that the cobar differential involving the cobrackets are absent exactly as in the case of $\cC=\coAss$. But then the statement is obtained by applying the symmetric algebra to the statement for $\cC=\coComm$.

Finally, the curved cases $\cC=\coAss^\theta$ and $\cC=\coP_n^\theta$ are reduced to the uncurved cases $\cC=\coAss$ and $\cC=\coP_n$ after passing to the associated gradeds.
\end{proof}

The statement has the following important corollaries.

\begin{cor}
The natural Hopf counital structures on the following cooperads are admissible:
\begin{itemize}
\item $\cC=\coAss$,

\item $\cC=\coAss^\theta$,

\item $\cC=\coComm$,

\item $\cC=\coP_n$,

\item $\cC=\coP_n^\theta$.
\end{itemize}
\label{cor:Hopfadmissible}
\end{cor}
\begin{proof}
Suppose $C,D,E$ are three (curved) conilpotent $\cC$-coalgebras with a morphism $D\rightarrow E$ such that $\Omega D\rightarrow \Omega E$ is a quasi-isomorphism. We have to show that
\[\Omega(C\otimes D)\to \Omega(C\otimes E)\]
is a quasi-isomorphism as well.

If $\cC=\coComm$, the statement follows from the commutative diagram
\[
\xymatrix{
\Omega(C\otimes D) \ar^{\sim}[d] \ar[r] & \Omega(C\otimes E) \ar^{\sim}[d] \\
\Omega(C)\oplus \Omega(D) \ar^{\sim}[r] & \Omega(C)\oplus \Omega(E)
}
\]
with vertical weak equivalences provided by Proposition \ref{prop:barlaxmonoidal}.

If $\cC=\coAss$ or $\cC=\coP_n$, the statement follows from the commutative diagram
\[
\xymatrix{
\Omega(C\otimes D) \ar^{\sim}[d] \ar[r] & \Omega(C\otimes E) \ar^{\sim}[d] \\
\Omega(C)\oplus \Omega(D)\oplus \Omega(C)\otimes \Omega(D)[-1] \ar^{\sim}[r] & \Omega(C)\oplus \Omega(E)\oplus \Omega(C)\otimes \Omega(E)[-1]
}
\]
with vertical weak equivalence given by the same proposition.

Finally, if $\cC=\coAss^\theta$ or $\cC=\coP_n^\theta$, the statement follows from the commutative diagram
\[
\xymatrix{
\Omega(C\otimes D) \ar^{\sim}[d] \ar[r] & \Omega(C\otimes E) \ar^{\sim}[d] \\
\Omega(C)\oplus \Omega(D) \ar^{\sim}[r] & \Omega(C)\oplus \Omega(E)
}
\]

In all these cases we are using the fact that the tensor product of complexes preserves quasi-isomorphisms.
\end{proof}

Consider the symmetric monoidal structure on the category of non-unital algebras $\alg_{\Ass}$ where the tensor product of $A_1$ and $A_2$ is
\[A_1\otimes A_2\oplus A_1\oplus A_2.\]
One also has a symmetric monoidal structure on the category of non-counital coalgebras $\coalg_{\coAss}$ where the tensor product of $C_1$ and $C_2$ is
\[C_1\otimes C_2\oplus C_1\oplus C_2.\]
Similarly, one introduces the symmetric monoidal structures on the categories $\alg_{\P_n}$ and $\coalg_{\coP_n}$.
Finally, consider the Cartesian symmetric monoidal structure on the category of Lie algebras $\alg_{\Lie}$.

\begin{cor}
The adjoint equivalences
\[\adj{\coalg_{\coAss}[\kos^{-1}]}{\alg_{\Ass}[\qis^{-1}]},\]
\[\adj{\coalg_{\coComm}[\kos^{-1}]}{\alg_{\Lie}[\qis^{-1}]}\]
and
\[\adj{\coalg_{\coP_n}[\kos^{-1}]}{\alg_{\P_n}[\qis^{-1}]}\]
are symmetric monoidal.
\end{cor}
\begin{proof}
By Corollary \ref{cor:Hopfadmissible} the cooperads $\coAss$, $\coComm$ and $\coP_n$ have Hopf admissible structures. Therefore, by Proposition \ref{prop:smlocalization} we obtain a lax symmetric monoidal equivalence
\[\adj{\coalg_{\coAss}[\kos^{-1}]}{\alg_{\Ass}[\qis^{-1}]}\]
and similarly for the other cooperads. But by Proposition \ref{prop:barlaxmonoidal} these are in fact symmetric monoidal functors.
\end{proof}

Therefore, we can use the relative symmetric monoidal category $(\coalg_{\coP_n}, \kos)$ as a model for the symmetric monoidal $\infty$-category $\ialg_{\P_n}$.

\section{Brace bar duality}

\subsection{Brace construction}

Suppose $\cC$ is a (curved) cooperad equipped with a Hopf counital structure (see Definition \ref{def:hopfcounital}). Let us briefly recall from \cite{CW} the definition of the associated operad of braces $\Br_\cC$. Its operations are parametrized by rooted trees with ``external'' vertices that are colored white in our pictures and ``internal'' vertices that are colored black. An external vertex with $r$ children is labeled by elements of $\cC^{\cu}(r)$ while an internal vertex with $r$ children is labeled by elements of $\overline{\cC}(r)[-1]$.

The coproduct on $\cC^{\cu}$ defines a morphism
\[\cC^{\cu}(r)\to \cC^{\cu}(r+r')\otimes \cC^{\cu}(1)^{\otimes r}\otimes \cC^{\cu}(0)^{\otimes r'}\to \cC^{\cu}(r+r'),\]
where in the second morphism we use $\cC^{\cu}(1)\rightarrow k$ given by the counit on the cooperad $\cC^{\cu}$ and $\cC^{\cu}(0)\cong \cC(0)\oplus k\rightarrow k$ given by projection on the second summand. The operadic composition is given by grafting rooted trees with labels obtained by applying the coproduct on $\cC^{cu}$ and combining different labels using the Hopf structure. The differential on internal vertices coincides with the cobar differential (in particular, if $\cC$ is curved, it contains an extra curving term); the differential of an external vertex splits it into an internal and an external vertex. We refer to \cite[Formula 8.14]{DW} for an explicit description of the differential.

\begin{figure}
\begin{tikzpicture}
\node[b] (v0) at (0, 0) {};
\node (v) at (0, -1) {};
\node[w] (v1) at (-1, 1) {$1$};
\node at (0, 1) {...};
\node[w] (v2) at (1, 1) {$k$};
\draw (v0) edge (v);
\draw (v0) edge (v1);
\draw (v0) edge (v2);
\end{tikzpicture}
\qquad
\begin{tikzpicture}
\node[w] (v0) at (0, 0) {$1$};
\node (v) at (0, -1) {};
\node[w] (v1) at (-1, 1) {$2$};
\node at (0, 1) {...};
\node[w] (v2) at (1, 1) {$k$};
\draw (v0) edge (v);
\draw (v0) edge (v1);
\draw (v0) edge (v2);
\end{tikzpicture}
\caption{Generating operations of $\Br_\cC$.}
\label{fig:brgenerating}
\end{figure}
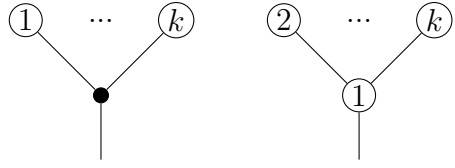

The operad $\Br_\cC$ is generated by trees shown in Figure \ref{fig:brgenerating}, where the leaves are labeled by the unit element of $\cC^{\cu}(0)$. For a $\Br_\cC$-algebra $A$ we denote the operation given by the corolla with an internal vertex by $m(c|x_1,...,x_r)$ where $c\in\overline{\cC}(r)[-1]$ and $x_i\in A$. The operation given by the corolla with an external vertex is denoted by $x\{c|y_1, ..., y_r\}$, where $c\in \cC^{\cu}(r)$. We denote by $\Omega\cC\rightarrow \Br_\cC$ the natural morphism which sends a generator in $\overline{\cC}[-1]$ to the corresponding corolla with an internal vertex.

The generating operations satisfy the following three relations.
\begin{enumerate}
\item (Associativity).

\begin{equation}
\begin{tikzpicture}
\node[w] (v0l) at (-1.7, 0) {$0$};
\node (vl) at (-1.7, -1) {};
\node[w] (v1l) at (-2.7, 1) {$z$};
\node at (-1.7, 1) {...};
\node[w] (v2l) at (-0.7, 1) {$z$};
\draw (v0l) edge (vl);
\draw (v0l) edge (v1l);
\draw (v0l) edge (v2l);

\node at (0, 0) {$\circ_0$};

\node[w] (v0r) at (1.7, 0) {$x$};
\node (vr) at (1.7, -1) {};
\node[w] (v1r) at (0.7, 1) {$y$};
\node at (1.7, 1) {...};
\node[w] (v2r) at (2.7, 1) {$y$};
\draw (v0r) edge (vr);
\draw (v0r) edge (v1r);
\draw (v0r) edge (v2r);

\node at (3.5, 0) {$=\sum\pm$};

\node[w] (v0rr) at (5, 0) {$x$};
\node (vrr) at (5, -1) {};
\node[w] (v1rr) at (4.1, 1) {$y$};
\node[w] (v2rr) at (4.7, 1) {$z$};
\node (v3rr) at (5.3, 1) {$...$};
\node[w] (v4rr) at (5.9, 1) {$y$};
\node[w] (v11rr) at (3.6, 2) {$z$};
\node (v12rr) at (4.1, 2) {$...$};
\node[w] (v13rr) at (4.6, 2) {$z$};
\node[w] (v41rr) at (5.9, 2) {$z$};
\draw (v0rr) edge (vrr);
\draw (v0rr) edge (v1rr);
\draw (v1rr) edge (v11rr);
\draw (v1rr) edge (v13rr);
\draw (v0rr) edge (v2rr);
\draw (v0rr) edge (v4rr);
\draw (v4rr) edge (v41rr);
\end{tikzpicture}
\label{eq:brace1}
\end{equation}

\item (Higher homotopies).

\begin{equation}
\begin{tikzpicture}
\node at (-3, 0) {$\d$};

\node[w] (v0l) at (-1.7, 0) {$0$};
\node (vl) at (-1.7, -1) {};
\node[w] (v1l) at (-2.7, 1) {$1$};
\node at (-1.7, 1) {...};
\node[w] (v2l) at (-0.7, 1) {$n$};
\draw (v0l) edge (vl);
\draw (v0l) edge (v1l);
\draw (v0l) edge (v2l);

\node at (0, 0) {$=\sum_{i=0}^n$};

\node[w] (v0r) at (1.7, 0) {$0$};
\node (vr) at (1.7, -1) {};
\node[w] (v1r) at (0.7, 1) {$1$};
\node at (1.7, 1) {...};
\node[w] (v2r) at (2.7, 1) {$n$};
\draw (v0r) edge (vr);
\draw (v0r) edge (v1r);
\draw (v0r) edge (v2r);

\node at (3, 0) {$\circ_i($};

\node[w] (v1a) at (3.8, 0) {$1$};
\node (va) at (3.8, -1) {};
\node[b] (v2a) at (3.8, 1) {};
\draw (v1a) edge (va);
\draw (v1a) edge (v2a);

\node at (4.5, 0) {$+$};

\node[b] (v1b) at (5.2, 0) {};
\node (vb) at (5.2, -1) {};
\node[w] (v2b) at (5.2, 1) {$1$};
\draw (v1b) edge (vb);
\draw (v1b) edge (v2b);

\node at (5.7, 0) {$)$};
\end{tikzpicture}
\label{eq:brace2}
\end{equation}

\item (Distributivity).

\begin{equation}
\begin{tikzpicture}
\node[w] (v0l) at (-1.7, 0) {$0$};
\node (vl) at (-1.7, -1) {};
\node[w] (v1l) at (-2.7, 1) {$y$};
\node at (-1.7, 1) {...};
\node[w] (v2l) at (-0.7, 1) {$y$};
\draw (v0l) edge (vl);
\draw (v0l) edge (v1l);
\draw (v0l) edge (v2l);

\node at (0, 0) {$\circ_0$};

\node[b] (v0r) at (1.7, 0) {};
\node (vr) at (1.7, -1) {};
\node[w] (v1r) at (0.7, 1) {$x$};
\node at (1.7, 1) {...};
\node[w] (v2r) at (2.7, 1) {$x$};
\draw (v0r) edge (vr);
\draw (v0r) edge (v1r);
\draw (v0r) edge (v2r);

\node at (3.5, 0) {$=\sum\pm$};

\node[b] (v0rr) at (5, 0) {};
\node (vrr) at (5, -1) {};
\node[w] (v1rr) at (4.1, 1) {$x$};
\node[w] (v2rr) at (4.7, 1) {$y$};
\node (v3rr) at (5.3, 1) {$...$};
\node[w] (v4rr) at (5.9, 1) {$x$};
\node[w] (v11rr) at (3.6, 2) {$y$};
\node (v12rr) at (4.1, 2) {$...$};
\node[w] (v13rr) at (4.6, 2) {$y$};
\node[w] (v41rr) at (5.9, 2) {$y$};
\draw (v0rr) edge (vrr);
\draw (v0rr) edge (v1rr);
\draw (v1rr) edge (v11rr);
\draw (v1rr) edge (v13rr);
\draw (v0rr) edge (v2rr);
\draw (v0rr) edge (v4rr);
\draw (v4rr) edge (v41rr);
\end{tikzpicture}
\label{eq:brace3}
\end{equation}
\end{enumerate}

Let us give some examples of the brace construction that we will use. Note that the second corolla in Figure \ref{fig:brgenerating} with $k=2$ gives rise to a pre-Lie structure on any $\Br_{\cC}$-algebra. However, in general the pre-Lie operation is not compatible with the differential.

The simplest example is the case $\cC=\bu$, the trivial cooperad. In this case $\overline{\cC}=0$ and hence the only operations are given by braces. However, since $\cC(n)=0$ for $n\geq 2$, we can only have vertices of valence 1 and 0, i.e. the operations of $\Br_{\bu}$ are parametrized by linear chains and hence $\Br_{\bu}(n)\cong k[S_n]$. It is immediate to see that the pre-Lie operation in this case gives rise to an associative multiplication.

\begin{prop}
One has an isomorphism of operads
\[\Ass\cong \Br_{\bu}.\]
\end{prop}

For $\cC=\coAss$ we obtain an $A_\infty$ structure of degree 1 together with degree $0$ brace operations $x\{y_1, ..., y_n\}$. Recall the brace operad $\Br$ introduced by Gerstenhaber and Voronov \cite{GV} and let $A_\infty$ be the operad controlling $A_\infty$ algebras.

\begin{prop}
We have a pushout of operads
\[
\xymatrix{
A_\infty \ar[r] \ar[d] & \Br_{\coAss}\{1\} \ar[d] \\
\Ass \ar[r] & \Br
}
\]
\end{prop}

Now consider $\cC=\coP_n$. Calaque--Willwacher \cite{CW} following some ideas of Tamarkin \cite{Ta1} introduced a morphism of operads
\begin{equation}
\Omega(\coP_{n+1}\{1\})\to \Br_{\coP_n},
\label{eq:Pnbrace}
\end{equation}
on generators by the following rule:
\begin{itemize}
\item The generators \[\underline{x_1...x_k}\in\coLie\{1-n\}(k)\subset \coP_{n+1}\{1\}(k)\] are sent to the tree drawn in Figure \ref{fig:centerhomotopymultiplication} with the root labeled by the element \[\underline{x_1...x_k}\in\coLie\{1-n\}(k)\subset \coP_n(k).\]

Here $\underline{x_1...x_k}$ is the image of the $k$-ary comultiplication under the projection \[\coAss\{1-n\}\rightarrow \coLie\{1-n\}.\]

\item The generator \[x_1\wedge x_2\in\coComm\{1\}(2)\subset \coP_{n+1}\{1\}(2)\] is sent to the linear combination of trees shown in Figure \ref{fig:centerhomotopybracket}.

\item The generators \[x_1\wedge \underline{x_2...x_k}\in\coP_{n+1}\{1\}(k)\] for $k>2$ are sent to the tree shown in Figure \ref{fig:centerhomotopyleibniz} with the root labeled by the element
\[\underline{x_2...x_k}\in\coLie\{1-n\}(k-1)\subset \coP_n(k-1).\]

\item The rest of the generators are sent to zero.
\end{itemize}

\begin{figure}
\begin{minipage}{.3\textwidth}
\centering
\begin{tikzpicture}
\node[b] (v0) at (0, 0) {};
\node (v) at (0, -1) {};
\node[w] (v1) at (-1, 1) {$1$};
\node at (0, 1) {...};
\node[w] (v2) at (1, 1) {$k$};
\draw (v0) edge (v);
\draw (v0) edge (v1);
\draw (v0) edge (v2);
\end{tikzpicture}
\caption{Image of $\underline{x_1...x_k}$.}
\label{fig:centerhomotopymultiplication}
\end{minipage}
\begin{minipage}{.3\textwidth}
\centering
\begin{tikzpicture}
\node[w] (v1) at (-1, 0) {$1$};
\node[w] (v2) at (-1, 1) {$2$};
\node (root) at (-1, -1) {};
\draw (v1) edge (v2);
\draw (v1) edge (root);
\node at (0, 0) {$-(-1)^n$};
\node[w] (v2) at (1, 0) {$2$};
\node[w] (v1) at (1, 1) {$1$};
\node (root) at (1, -1) {};
\draw (v2) edge (v1);
\draw (v2) edge (root);
\end{tikzpicture}
\caption{Image of $x_1\wedge x_2$.}
\label{fig:centerhomotopybracket}
\end{minipage}
\begin{minipage}{.3\textwidth}
\centering
\begin{tikzpicture}
\node[w] (v0) at (0, 0) {$1$};
\node (v) at (0, -1) {};
\node[w] (v1) at (-1, 1) {$2$};
\node at (0, 1) {...};
\node[w] (v2) at (1, 1) {$k$};
\draw (v0) edge (v);
\draw (v0) edge (v1);
\draw (v0) edge (v2);
\end{tikzpicture}
\caption{Image of $x_1\wedge\underline{x_2...x_k}$}
\label{fig:centerhomotopyleibniz}
\end{minipage}
\end{figure}

\begin{prop}[Calaque--Wilwacher]
The morphism of operads
\[\Omega(\coP_{n+1}\{1\})\to \Br_{\coP_n}\]
is a quasi-isomorphism.
\label{prop:Pnbrace}
\end{prop}
Note that we have a quasi-isomorphism of operads
\[\Omega(\coP_{n+1}\{n+1\})\to \P_{n+1}\]
and hence we get a zig-zag of quasi-isomorphisms between the operads $\Br_{\coP_n}\{n\}$ and $\P_{n+1}$.

One can similarly define a morphism of operads
\[\Omega(\coP^{\theta}_{n+1}\{1\})\rightarrow \Br_{\coP_n^\theta}\]
in the following way:
\begin{itemize}
\item The generators $y\in \coLie^\theta\{1-n\}\subset \coP^\theta_{n+1}\{1\}$ are sent to the tree drawn in Figure \ref{fig:centerhomotopymultiplication} with the root labeled by $x\in\coLie^\theta\{1-n\}\subset \coP^\theta_n$.

\item The generator \[x_1\wedge x_2\in\coComm\{1\}(2)\subset \coP^\theta_{n+1}\{1\}(2)\] is sent to the linear combination of trees shown in Figure \ref{fig:centerhomotopybracket}.

\item Suppose $y\in\coLie^\theta\{1-n\}(k-1)\subset \coP^\theta_{n+1}\{1\}(k-1)$ for $k>2$ and let $x\wedge y$ be the image of $y$ under $\coP^\theta_{n+1}\{1\}(k-1)\rightarrow \coP^\theta_{n+1}\{1\}(k)$.

The generators $x\wedge y\in\coP^\theta_{n+1}\{1\}(k)$ are sent to the tree shown in Figure \ref{fig:centerhomotopyleibniz} with the root labeled by the element $y\in\coLie^\theta\{1-n\}(k-1)\subset\coP^\theta_n(k-1)$.

\item The rest of the generators are sent to zero.
\end{itemize}

The following statement is a version of Proposition \ref{prop:Pnbrace}.
\begin{prop}
The morphism of operads
\[\Omega(\coP^\theta_{n+1}\{1\})\rightarrow \Br_{\coP^\theta_n}\]
is a quasi-isomorphism.
\end{prop}
This proposition implies that we have a zig-zag of weak equivalences between the operads $\Br_{\coP^\theta_n}\{n\}$ and $\P_{n+1}^{\un}$.

Finally, consider the case $\cC=\coComm$. By construction we have a morphism of operads
\[\Lie\rightarrow \Br_\cC\]
given by sending the Lie bracket to the combination
\[
\begin{tikzpicture}
\node[w] (v1l) at (-1, 0) {$1$};
\node[w] (v2l) at (-1, 1) {$2$};
\node (v0l) at (-1, -1) {};
\draw (v1l) edge (v0l);
\draw (v1l) edge (v2l);

\node at (0, 0) {$-$};

\node[w] (v1r) at (1, 1) {$1$};
\node[w] (v2r) at (1, 0) {$2$};
\node (v0r) at (1, -1) {};
\draw (v2r) edge (v0r);
\draw (v2r) edge (v1r);
\end{tikzpicture}
\]

\begin{prop}
The morphism of operads $\Lie\rightarrow \Br_\coComm$ is a quasi-isomorphism.
\label{prop:Liebrace}
\end{prop}
\begin{proof}
Introduce a grading on $\coP_n$ by setting the cobracket to be of weight $1$ and the comultiplication of weight $0$. In this way $\coP_n^{\cu}$ becomes a graded Hopf cooperad, i.e. a cooperad in graded commutative dg algebras. This induces a grading on the brace construction $\Br_{\coP_n}$, where the weight of a tree is given by the sum of the weights of the labels. The morphism \eqref{eq:Pnbrace} is compatible with the gradings if we introduce the grading on $\coP_{n+1}$ where again the cobracket has weight $1$ and the comultiplication has weight $0$.

Let $L_\infty = \Omega(\coComm\{1\})$ be the operad controlling $L_\infty$ algebras. Passing to weight $0$ components and using Proposition \ref{prop:Pnbrace}, we obtain a quasi-isomorphism
\[L_\infty\to \Br_\coComm\]
which by construction factors as
\[L_\infty\stackrel{\sim}\to \Lie\rightarrow \Br_{\coComm}\]
and the claim follows.
\end{proof}

\subsubsection{Aside: the operad $\Br_{\coComm}$}
\label{sect:brcocomm}

Let us explain the role the operad $\Br_{\coComm}$ plays in Lie theory (see also Proposition \ref{prop:LieadditivityU}).

Suppose that $\mathfrak{g}$ is a Lie algebra with a lift of the structure to a $\Br_\coComm$-algebra. That is, $\mathfrak{g}$ has a pre-Lie structure $x\circ y$ and degree $L_\infty$ brackets $\{x_1,...,x_n\}$ such that
\begin{align*}
[x, y] &= x\circ y - (-1)^{|x||y|}y\circ x \\
\{x, y\} &= \d(x\circ y) - (\d x)\circ y - (-1)^{|x|} x\circ \d y,
\end{align*}
where $[x,y]$ is the original degree $0$ Lie bracket. Note that the degree $1$ $L_\infty$ brackets are uniquely determined from the pre-Lie operation. Then the $\Br_{\coComm}$-algebra structure allows one to replace the morphism of Lie algebras $0\rightarrow \mathfrak{g}$ by a fibration in the following way.

Consider the complex
\begin{equation}
\widetilde{\mathfrak{g}}=\mathfrak{g}\oplus\mathfrak{g}[-1]
\label{eq:Lieresolution}
\end{equation}
with the identity differential. We define an $L_\infty$ algebra structure on $\widetilde{\mathfrak{g}}$ as follows:
\begin{itemize}
\item The bracket on the first term is the original bracket $[-,-]$.

\item The $L_\infty$ structure on the second term is given by the operations $\{-,...,-\}$.

\item The $L_\infty$ brackets $[x, \s y_1, ..., \s y_n]$ where $x\in\g$ and $\s y_i\in\g[-1]$ land in $\g[-1]$ and are given by the symmetric braces $\g\otimes \Sym(\g)\rightarrow \g$.
\end{itemize}

It is immediate that $\widetilde{\mathfrak{g}}\rightarrow \mathfrak{g}$ is a fibration of $L_\infty$ algebras (i.e. it is a degreewise surjective morphism) and, moreover, that the morphism $0\rightarrow \widetilde{\mathfrak{g}}$ is a quasi-isomorphism.

In particular, we see that the $L_\infty$ algebra structure on
\[\Omega\mathfrak{g} = 0\oplus_{\mathfrak{g}} 0\cong \mathfrak{g}[-1]\]
is given by the degree 1 $L_\infty$ brackets $\{-, ..., -\}$.

\begin{example}
Let $A$ be a commutative dg algebra over a field $k$ and denote by $\T_A=\Der_k(A, A)$ the complex of derivations. Suppose $\nabla$ is a flat torsion-free connection on the underlying graded algebra, i.e. it defines a morphism of graded vector spaces
\[\nabla\colon \T_A\otimes_k \T_A\rightarrow \T_A\]
such that
\begin{align}
[v, w] &= \nabla_v w - (-1)^{|v||w|}\nabla_w v \label{eq:torsionfree} \\
\nabla_{[v, w]} &= \nabla_v \nabla_w - (-1)^{|v||w|}\nabla_w \nabla_v \label{eq:flat}
\end{align}
for two vector fields $v,w\in\T_A$.

Then the Lie algebra structure on $\T_A$ given by the commutator of derivations is lifted to a $\Br_\coComm$-algebra structure, where the pre-Lie structure is given by the connection:
\[v\circ w = \nabla_v w.\]
Indeed, equation \eqref{eq:torsionfree} implies that $\nabla_v w$ lifts the Lie bracket of vector fields and equation \eqref{eq:flat} implies that it is indeed a pre-Lie bracket.

Therefore, we see that the $L_\infty$ structure on
\[\Omega \T_A = \T_A[-1]\]
is given by
\[\{v, w\} = \d(\nabla_v w) - \nabla_{\d v} w - (-1)^{|v|}\nabla_v (\d w).\]
In this way we discover exactly the Atiyah bracket of vector fields as defined by Kapranov, see \cite[Section 2.5]{Ka}.
\end{example}

\subsection{Additivity for brace algebras}
\label{sect:bracebarcobar}

Recall that we have a morphism of operads $\Omega \cC\rightarrow \Br_\cC$. In particular, a brace algebra has a bar complex which is a $\cC$-coalgebra. Now we are going to introduce an associative multiplication on the bar complex making the diagram
\[
\xymatrix{
\alg_{\Br_{\cC}} \ar[r] \ar@{-->}^{\B}[d] & \alg_{\Omega\cC} \ar^{\B}[d] \\
\alg(\coalg_{\cC}) \ar[r] & \coalg_{\cC}
}
\]
commute.

\begin{remark}
Consider a bialgebra $C\in \alg(\coalg_{\cC^{\cu}}^{\coaug})$. The unit of the symmetric monoidal structure on $\coalg_{\cC^{\cu}}^{\coaug}$ is given by $k$ with the identity coaugmentation. Therefore, the unit morphism for $C$ is a morphism of coaugmented $\cC^{\cu}$-coalgebras
\[k\rightarrow C.\]
Compatibility with the coaugmentation on $C$ implies that the unit morphism $k\rightarrow C$ coincides with the coaugmentation $k\rightarrow C$.
\end{remark}

Let $A$ be a $\Br_{\cC}$-algebra and consider $\cC^{\cu}(A)$, the cofree conilpotent $\cC^{\cu}$-coalgebra, equipped with the bar differential. It is naturally coaugmented using the decomposition \[\cC^{\cu}(0)\cong \cC(0)\oplus k.\]

We introduce the unit on $\cC^{\cu}(A)$ to be given by the coaugmentation $k\rightarrow \cC^{\cu}(A)$. The multiplication
\[\cC^{\cu}(A)\otimes \cC^{\cu}(A)\rightarrow \cC^{\cu}(A)\]
is uniquely specified on the cogenerators by a morphism
\[\cC^{\cu}(A)\otimes \cC^{\cu}(A)\rightarrow A.\]
In turn, it is defined via the composite
\[\cC^{\cu}(A)\otimes \cC^{\cu}(A)\rightarrow A\otimes \cC^{\cu}(A)\rightarrow A,\]
where the first morphism is induced by the counit on $\cC^{\cu}$ and the second morphism is given by braces, i.e. the morphism
\[\cC^{\cu}(n)\otimes A\otimes A^{\otimes n}\rightarrow A\]
is given by applying the second corolla in Figure \ref{fig:brgenerating} with the root labeled by the element of $\cC^{\cu}(n)$ and the leaves labeled by the unit $k\rightarrow \cC^{\cu}(0)$. The following statement is shown in \cite[Proposition 3.4]{MS}.

\begin{prop}
This defines a unital dg associative multiplication on $\cC^{\cu}(A)$ compatible with the $\cC^{\cu}$-coalgebra structure.
\end{prop}

For a $\Br_{\cC}$-algebra $A$ we denote by $\B A=\cC^{\cu}(A)$ the bar complex equipped with the bar differential and the above associative multiplication. This defines a functor
\[\B\colon \alg_{\Br_{\cC}}\to \alg(\coalg_{\cC^{\cu}}^{\coaug}).\]

Now suppose the Hopf unital structure on $\cC$ is admissible. Then the symmetric monoidal structure on $\coalg^{\coaug}_{\cC^{\cu}}$ preserves weak equivalences and hence $\coalg^{\coaug}_{\cC^{\cu}}[\kos^{-1}]$ is a symmetric monoidal $\infty$-category.

Since the bar functor preserves weak equivalences, the composite functor
\[\alg_{\Br_{\cC}}[\qis^{-1}]\to \alg(\coalg_{\cC})[\kos^{-1}]\to \ialg(\coalg_{\cC}[\kos^{-1}])\]
gives rise to the additivity functor
\begin{equation}
\add\colon \ialg_{\Br_{\cC}}\to \ialg(\ialg_{\Omega\cC}).
\label{eq:additivityfunctor}
\end{equation}

We do not know if the functor \eqref{eq:additivityfunctor} is an equivalence in general, but we prove that it is the case for Lie and Poisson algebras. Note that in the case $\cC=\coAss$ we obtain a functor
\[\ialg_{\E_2}\cong \ialg_{\Br}\to \ialg(\ialg)\]
which we expect coincides with the Dunn--Lurie equivalence \cite[Theorem 5.1.2.2]{HA}.

\subsection{Additivity for Lie algebras}
\label{sect:Lieadditivity}

In this section we work out how the additivity functor \eqref{eq:additivityfunctor} looks like for Lie algebras and prove that it is an equivalence. We consider the cooperad $\cC=\coComm$.

We have a functor
\[\U\colon \alg_{\Lie}\to \alg(\coalg^{\coaug}_{\coComm^{\cu}})\]
which sends a Lie algebra to its universal enveloping algebra. The following is proved e.g. in \cite[Theorem 3.8.1]{Ca}.

\begin{thm}[Cartier--Milnor--Moore]
The universal enveloping algebra
\[\U\colon \alg_{\Lie}\to \alg(\coalg^{\coaug}_{\coComm^{\cu}})\]
is an equivalence of categories.
\label{thm:CartierMilnorMoore}
\end{thm}

\begin{lm}
The equivalence of categories
\[\U\colon \alg_{\Lie}\stackrel{\sim}\to \alg(\coalg^{\coaug}_{\coComm^{\cu}})\]
preserves weak equivalences.
\end{lm}
\begin{proof}
Recall that weak equivalences in $\alg(\coalg^{\coaug}_{\coComm^{\cu}})$ are created by the forgetful functor to $\coalg_{\coComm}$. Given a dg Lie algebra $\g$, the PBW theorem gives an identification of cocommutative coalgebras $\U(\g)\cong \Sym(\g)$, hence for a morphism of dg Lie algebras $\g_1\rightarrow \g_2$ we have a commutative diagram
\[
\xymatrix{
\Omega \U(\g_1) \ar^{\sim}[d] \ar[r] & \Omega \U(\g_2) \ar^{\sim}[d] \\
\g_1[-1] \ar[r] & \g_2[-1]
}
\]
where the vertical morphisms are quasi-isomorphisms. Therefore, the bottom morphism is a quasi-isomorphism iff the top morphism is a quasi-isomorphism.
\end{proof}

We get two functors
\[\alg_{\Br_{\coComm}}\to \alg(\coalg_{\coComm})\]
where one is given by the brace bar construction $\B$ and the other one is given by the composite
\[\alg_{\Br_{\coComm}}\xrightarrow{\forget} \alg_{\Lie}\xrightarrow{\U} \alg(\coalg_{\coComm})\]
where the forgetful functor is given by Proposition \ref{prop:Liebrace}. In fact, these are equivalent.

\begin{prop}
There is a natural weak equivalence
\[\U\circ\forget\stackrel{\sim}\to \B\]
of functors
\[\alg_{\Br_{\coComm}}\to \alg(\coalg_{\coComm}).\]
\label{prop:LieadditivityU}
\end{prop}
\begin{proof}
Suppose $\g$ is a $\Br_{\coComm}$-algebra. In particular, $\g$ is a pre-Lie algebra and we have to produce a natural isomorphism of dg cocommutative bialgebras
\[\U(\g)\stackrel{\sim}\to\Sym(\g),\]
where $\Sym(\g)$ is equipped with the associative product using the pre-Lie structure and the differential uses the shifted $L_\infty$ brackets on $\g$.

The morphism $\U(\g)\rightarrow \Sym(\g)$ is uniquely determined by a map
\[\g\rightarrow \Sym(\g)\]
on generators which we define to be the obvious inclusion. We refer to \cite[Theorem 2.12]{OG} for the claim that it extends to an isomorphism of cocommutative bialgebras $\U(\g)\rightarrow \Sym(\g)$. The compatibility with the differential is obvious as the differential on $\g\cong \coComm(1)\otimes \g$ is simply given by the differential on $\g$.
\end{proof}

The $\infty$-category of Lie algebras is pointed, i.e. the initial and final objects coincide. Therefore, we can consider the loop functor
\[\Omega\colon \ialg_{\Lie}\to \ialg(\ialg_{\Lie})\]
given by sending a Lie algebra $\g$ to its loop object $0\times_{\g} 0$. More explicitly, since the monoidal structure on $\ialg_{\Lie}$ is Cartesian, by \cite[Proposition 4.1.2.10]{HA} we can identify $\ialg(\ialg_{\Lie})$ with the $\infty$-category of Segal monoids, i.e. simplicial objects $M_\bullet$ of $\ialg_{\Lie}$ such that $M_0$ is contractible and the natural maps $M_n\rightarrow M_1\times_{M_0}\times ...\times_{M_0} M_1$ are equivalences. Under this identification the loop object of $\g$ is defined to be the simplicial object underlying the Cech nerve of $0\rightarrow \g$:
\[
\xymatrix{
0 & 0\times_{\g} 0 \ar@<.5ex>[l] \ar@<-.5ex>[l] & 0\times_{\g} 0\times_{\g} 0 \ar@<.7ex>[l] \ar[l] \ar@<-.7ex>[l] & ...
}
\]

Its left adjoint is the classifying space functor
\[\B\colon \ialg(\ialg_{\Lie})\to \ialg_{\Lie}\]
which sends a Segal monoid in Lie algebras to its geometric realization. The following is proved in \cite[Lemma 5.3]{To} and \cite[Corollary 2.7.2]{GH}.

\begin{prop}
The adjunction
\[\adj{\B\colon \ialg(\ialg_{\Lie})}{\ialg_{\Lie}\colon \Omega}\]
is an equivalence of symmetric monoidal $\infty$-categories.
\label{prop:Lieloopequivalence}
\end{prop}

Since the operads $\Br_{\coComm}$ and $\Lie$ are quasi-isomorphic, the additivity functor \eqref{eq:additivityfunctor} becomes
\[\add\colon \ialg_{\Lie}\to \ialg(\ialg_{\Lie}).\]
Observe that now we have constructed two functors $\ialg_{\Lie}\to \ialg(\ialg_{\Lie})$: the loop functor $\Omega$ and the additivity functor $\add$.

\begin{prop}
The additivity functor $\add\colon \ialg_{\Lie}\rightarrow \ialg(\ialg_{\Lie})$ is equivalent to the loop functor $\Omega$.
\label{prop:LieadditivityOmega}
\end{prop}
\begin{proof}
Since the classifying space functor $\B$ is an inverse to the loop functor $\Omega$ by Proposition \ref{prop:Lieloopequivalence}, we have to prove that we have an equivalence $\B\circ \add\cong \id$ of functors $\ialg_{\Lie}\rightarrow \ialg_{\Lie}$. By Proposition \ref{prop:LieadditivityU} we can write $\add$ as the composite
\begin{align*}
\alg_{\Lie}[\qis^{-1}]&\stackrel{\U}\to \alg(\coalg_{\coComm}) [\kos^{-1}] \\
&\longrightarrow \ialg(\coalg_{\coComm}[\kos^{-1}]) \\
&\cong \ialg(\ialg_{\Lie}),
\end{align*}
where the last equivalence is given by the Chevalley--Eilenberg complex which realizes the bar construction for Lie algebras.

Therefore, the statement will follow once we show that the composite
\begin{align}
\alg_{\Lie}[\qis^{-1}]&\stackrel{\U}\to \alg(\coalg_{\coComm}) [\kos^{-1}] \nonumber \\
&\longrightarrow \ialg(\coalg_{\coComm}[\kos^{-1}]) \nonumber \\
&\stackrel{\B}\longrightarrow \coalg_{\coComm}[\kos^{-1}]
\label{eq:Lieadditivitycomposite}
\end{align}
is equivalent to the Chevalley-Eilenberg complex functor
\[\C_\bullet\colon \alg_{\Lie}[\qis^{-1}]\to \coalg_{\coComm}[\kos^{-1}].\]

For an algebra $A$ in a symmetric monoidal $\infty$-category $\cC$, a right $A$-module $M$ and a left $A$-module $N$ let us denote by $\Bar_\bullet(M, A, N)$ the simplicial object of $\cC$ underlying the two-sided bar construction whose $n$-simplices are given by $M\otimes A^{\otimes n}\otimes N$. The composite \eqref{eq:Lieadditivitycomposite} applied to a Lie algebra $\g$ is then by definition $|\Bar_\bullet(k, \U\g, k)|\cong k\otimes^{\bL}_{\U\g} k$ and is a cocommutative coalgebra since $k$ and $\U\g$ are.

For a dg Lie algebra $\g$ we define following \cite[Section 2.2]{DAGX} the Lie algebra $\Cn(\g)$ which as a graded vector space is $\Cn(\g)=\g\oplus \g[1]$ equipped with the following dg Lie algebra structure:
\begin{itemize}
\item The differential is given by the identity differential from the second term to the first term and the internal differentials on the two summands of $\g$.

\item The Lie bracket on the first term is the original Lie bracket on $\g$.

\item The Lie bracket between $\s^{-1} x\in\g[1]$ and $y\in\g$ lands in $\g[1]$ and is given by
\[[\s^{-1} x, y] = \s^{-1}[x, y].\]

\item The Lie bracket on the last term is zero.
\end{itemize}

We have a quasi-isomorphism of $\g$-modules $0\rightarrow \Cn(\g)$. Therefore, after taking the universal enveloping algebra we obtain a weak equivalence of left $\U\g$-modules in cocommutative coalgebras
\[k\to \U(\Cn(\g))\cong \U\g\otimes \Sym(\g[1]),\]
where on the right we have used the PBW isomorphism.

Therefore, we have a weak equivalence of Segal monoids
\[\Bar_\bullet(k, \U\g, k)\rightarrow \Bar_\bullet(k, \U\g, \U(\Cn(\g))).\]
The homotopy colimit of $\Bar_\bullet(k, \U\g, \U(\Cn(\g)))$ is a strict colimit since $\U(\Cn(\g))$ is a semi-free left $\U\g$-module. But its strict colimit is
\[k\otimes_{\U\g}\U(\Cn(\g))\cong\C_\bullet(\g)\]
as dg cocommutative coalgebras. Therefore, we obtain a natural equivalence $|\Bar_\bullet(k, \U\g, k)|\cong \C_\bullet(\g)$ and the claim follows.
\end{proof}

Combining the previous proposition with Proposition \ref{prop:Lieloopequivalence} we obtain the following corollary.

\begin{cor}
The additivity functor
\[\add\colon \ialg_{\Lie}\to \ialg(\ialg_{\Lie})\]
is an equivalence of symmetric monoidal $\infty$-categories.
\label{cor:Lieadditivity}
\end{cor}

Note that the underlying Lie algebra of $\ialg(\ialg_{\Lie})$ is canonically trivial. Indeed, let $\iCh\to \ialg_{\Lie}$ be the functor which sends a complex $\g$ to the trivial Lie algebra $\g[-1]$.

\begin{prop}
The diagram
\[
\xymatrix{
\ialg_{\Lie} \ar^-{\add}[r] \ar[d] & \ialg(\ialg_{\Lie}) \ar[d] \\
\iCh \ar[r] & \ialg_{\Lie}
}
\]
is commutative.
\label{prop:underlyingLietrivial}
\end{prop}
\begin{proof}
Indeed, the PBW theorem implies that we have a commutative diagram of categories
\[
\xymatrix{
\alg_{\Lie} \ar^-{\U}[r] \ar[d] & \alg(\coalg_{\coComm}) \ar[d] \\
\Ch \ar^-{\Sym}[r] & \coalg_{\coComm}
}
\]
i.e. the underlying cocommutative coalgebra of $\U(\g)$ is isomorphic to $\Sym(\g)$. But $\Sym(\g)$ coincides with the bar construction of a trivial Lie algebra $\g[-1]$ and the claim follows.
\end{proof}

\subsection{Additivity for Poisson algebras}
\label{sect:poissonadditivity}

This section is devoted to an explicit description of the additivity functor \eqref{eq:additivityfunctor} in the case of $\P_n$-algebras and to showing that it is a symmetric monoidal equivalence of $\infty$-categories.

Consider the cooperad $\cC=\coP_n$. Since we have a zig-zag of quasi-isomorphisms between the operads $\P_{n+1}$ and $\Br_{\coP_n}\{n\}$, the additivity functor \eqref{eq:additivityfunctor} becomes
\begin{equation}
\add\colon \ialg_{\P_{n+1}}\to \ialg(\ialg_{\P_n}).
\label{eq:Poissonadditivity}
\end{equation}

Now we are going to give a different perspective on this functor closer to Tamarkin's papers \cite{Ta1} and \cite{Ta2}. This new perspective will elucidate the formulas for the morphism \eqref{eq:Pnbrace} and allow us to show that the additivity functor is a symmetric monoidal equivalence.

We are going to introduce yet another version of the bar-cobar duality for Poisson algebras, this time the dual object will be a Lie bialgebra.

\begin{defn}
An \defterm{$n$-shifted Lie bialgebra} is a dg Lie algebra $\g$ together with a degree $-n$ Lie coalgebra structure $\delta\colon \g\rightarrow \g\otimes \g[-n]$ satisfying the cocycle relation
\begin{equation}
\delta([x, y]) = (\ad_x\otimes \id + \id \otimes \ad_x)\delta(y) - (-1)^{n+|x||y|} (\ad_y\otimes \id + \id\otimes \ad_y)\delta(x).
\label{eq:Liebialgcocycle}
\end{equation}
\label{def:liebialg}
\end{defn}

We will say an $n$-shifted Lie bialgebra is \defterm{conilpotent} if the underlying Lie coalgebra is so and we denote the category of $n$-shifted conilpotent Lie bialgebras by $\bialg_{\Lie_n}$. Weak equivalences in $\bialg_{\Lie_n}$ are created by the forgetful functor
\[\bialg_{\Lie_n}\to \coalg_{\coLie}.\]

The commutative bar-cobar adjunction
\[\adj{\Omega\colon \coalg_{\coLie}}{\alg_{\Comm}\colon \B}\]
extends to a bar-cobar adjunction
\begin{equation}
\adj{\Omega\colon \bialg_{\Lie_{n-1}}}{\alg_{\P_{n+1}}\colon \B}
\label{eq:PnLiebialg}
\end{equation}
so that the diagram
\[
\xymatrix{
\alg_{\P_{n+1}} \ar@<.5ex>[r] \ar[d] & \bialg_{\Lie_{n-1}} \ar[d] \ar@<.5ex>[l] \\
\alg_{\Comm} \ar@<.5ex>[r] & \coalg_{\coLie} \ar@<.5ex>[l]
}
\]
commutes where the vertical functors are the forgetful functors.

Explicitly, if $A$ is a $\P_{n+1}$-algebra, consider $\g=\coLie(A[1])[n-1]$, the cofree conilpotent shifted Lie coalgebra equipped with the bar (i.e. Harrison) differential. By the cocycle equation a Lie bracket on $\g$ is uniquely determined by its projection to cogenerators and the map
\[\coLie(A[1])[n-1]\otimes \coLie(A[1])[n-1]\to \coLie(A[1])[n-1]\to A[n]\]
is defined to be the Lie bracket
\[A[n]\otimes A[n]\to A[n].\]
The Jacobi identity is obvious. Compatibility of the bracket on $\g$ with the bar differential follows from the Leibniz rule for $A$.

Conversely, if $\g$ is an $(n-1)$-shifted Lie bialgebra, consider $A=\Sym(\g[-n])$ equipped with the cobar differential using the Lie coalgebra structure on $\g$. The Lie bracket on $A$ by the Leibniz rule is defined on generators to be the Lie bracket on $\g$. The compatibility of the cobar differential on $A$ with the Lie structure can be checked on generators where it coincides with the cocycle equation \eqref{eq:Liebialgcocycle}.

By the definition of weak equivalences it is clear that the adjunction \eqref{eq:PnLiebialg} induces an adjoint equivalence on $\infty$-categories since the unit and counit of the adjunction are weak equivalences after forgetting down to commutative algebras and Lie coalgebras.

By the Cartier--Milnor--Moore theorem (Theorem \ref{thm:CartierMilnorMoore}) the universal enveloping algebra functor induces an equivalence of categories
\[\U\colon \alg_{\Lie}\stackrel{\sim}\to \alg(\coalg_{\coComm}).\]

If $\g$ is an $(n-1)$-shifted Lie bialgebra, we can define a $\P_n$-coalgebra structure on $\U(\g)$ as follows. By construction $\U(\g)$ is a cocommutative bialgebra and we define the cobracket on the generators to be the cobracket on $\g$. The coproduct on $\U(\g)$ is conilpotent and the cobracket on $\U(\g)$ is conilpotent if the cobracket on $\g$ is so. In this way we construct a commutative diagram
\[
\xymatrix{
\bialg_{\Lie_{n-1}} \ar^-{\sim}[r] \ar[d] & \alg(\coalg_{\coP_n}) \ar[d] \\
\alg_{\Lie} \ar^-{\sim}[r] & \alg(\coalg_{\coComm})
}
\]

Note that if $A\in\alg(\coalg^{\coaug}_{\coP^{\cu}_n})$ has an associative multiplication and a compatible $\P_n$-coalgebra structure, then the space of primitive elements is closed under the cobracket by the Leibniz rule for the $\P_n$-coalgebra structure. Thus, the inverse functor in both cases is simply given by the functor of primitive elements.

Now we are going to show that the brace bar construction for $\cC=\coP_n$ is compatible with the bar-cobar duality between $\P_{n+1}$-algebras and $(n-1)$-shifted Lie bialgebras in the following sense.

\begin{prop}
The composite functor
\begin{align*}
\alg_{\Br_{\coP_n}}&\to \alg(\coalg_{\coP_n})\\
&\to \bialg_{\Lie_{n-1}}\\
&\to \alg_{\P_{n+1}}\\
&\stackrel{\forget}\to \alg_{\Omega(\coP_{n+1}\{1\})}
\end{align*}
is weakly equivalent to the forgetful functor
\[\alg_{\Br_{\coP_n}}\to \alg_{\Omega(\coP_{n+1}\{1\})}\]
given by equation \eqref{eq:Pnbrace}.
\label{prop:bracepoissoncompatible}
\end{prop}
\begin{proof}
Let $A$ be a $\Br_{\coP_n}$-algebra. The functor
\[\alg_{\Br_{\coP_n}}\to \alg(\coalg_{\coP_n})\]
sends $A$ to $\coP_n(A)$ equipped with the bar differential using the homotopy $\P_n$-algebra structure on $A$.

We can identify $\coP_n\cong \coComm\circ \coLie\{1-n\}$ as symmetric sequences and hence after passing to primitives in $\coP_n(A)$ we obtain $\g=\coLie(A[1-n])[n-1]$ equipped with the bar differential using the homotopy commutative algebra structure on $A$. The Lie bracket on $\g$ is obtained by antisymmetrizing the associative multiplication on $\coP_n(A)$ and from the explicit description of the multiplication on $\coP_n(A)$ given in Section \ref{sect:bracebarcobar} we see that the projection of the bracket on the cogenerators
\[\coLie(A[1-n])[n-1]\otimes \coLie(A[1-n])[n-1]\to A\]
is defined by the morphism
\[A\otimes \coLie(A[1-n])[n-1]\to A\]
given by the brace operations $x\{c|y_1,..., y_m\}$ where $x,y_i\in A$ and $c\in\coLie\{1-n\}(m)$.

We conclude that $\coLie(A[1-n])[n-1]$ is the Koszul dual of a homotopy $\P_{n+1}$-algebra $A$ whose homotopy commutative multiplication is encoded in the differential on $\g$ which comes from the homotopy commutative multiplication in the $\Br_{\coP_n}$-algebra structure. The rest of the homotopy $\P_{n+1}$ structure on $A$ coincides with the one given by the morphism \eqref{eq:Pnbrace} by inspection.
\end{proof}

By the previous proposition the two functors of $\infty$-categories
\[\ialg_{\P_{n+1}}\to \ialg(\ialg_{\P_n})\]
given either by localizing the brace bar functor for $\cC=\coP_n$ or by localizing the bar-cobar duality between $\P_{n+1}$-algebras and shifted Lie bialgebras coincide.

\begin{prop}
The functor
\[\ialg_{\P_{n+1}}\to \ialg(\ialg_{\P_n})\]
has a natural symmetric monoidal structure.
\label{prop:Poissonadditivitysm}
\end{prop}
\begin{proof}
The functor $\Omega\colon \bialg_{\Lie_{n-1}}\rightarrow \alg_{\P_{n+1}}$ is symmetric monoidal, so its right adjoint $\B\colon \alg_{\P_{n+1}}\rightarrow \bialg_{\Lie_{n-1}}$ has a lax symmetric monoidal structure. Moreover, by Proposition \ref{prop:barlaxmonoidal}, $\B$ becomes strictly symmetric monoidal after localization. Finally, the universal enveloping algebra functor
\[\U\colon \bialg_{\Lie_{n-1}}\to \alg(\coalg_{\coP_n})\]
has an obvious symmetric monoidal structure and the claim follows.
\end{proof}

As a corollary, we get a sequence of functors
\[\ialg_{\P_{n+2}}\to \ialg(\ialg_{\P_{n+1}})\to \ialg(\ialg(\ialg_{\P_n})).\]
But $\ialg(\ialg(\cC))\cong \ialg_{\E_2}(\cC)$ for any symmetric monoidal $\infty$-category $\cC$ by the Dunn--Lurie additivity theorem \cite[Theorem 5.1.2.2]{HA}. Iterating this construction, we get a symmetric monoidal functor
\[\ialg_{\P_{n+m}}\to \ialg_{\E_m}(\ialg_{\P_n}).\]

The additivity functor \eqref{eq:Poissonadditivity} interacts in the obvious way with the commutative and the Lie structures on a $\P_n$-algebra as shown by the next three propositions.

\begin{prop}
The diagram
\[
\xymatrix{
\ialg_{\P_{n+1}} \ar[r] \ar[d] & \ialg(\ialg_{\P_n}) \ar[d] \\
\ialg_{\Comm} & \ialg_{\P_n} \ar[l]
}
\]
is commutative.
\label{prop:CommPoissoncompatible}
\end{prop}
\begin{proof}
The claim immediately follows from the commutative diagram of operads
\[
\xymatrix{
\Omega(\coP_{n+1}\{1\}) \ar[r] & \Br_{\coP_n} \\
\Omega(\coLie\{1-n\}) \ar[u] \ar[r] & \Omega(\coP_n) \ar[u]
}
\]
\end{proof}

Let us denote by
\[\overline{\Sym}\colon \alg_{\Lie}\to \alg_{\P_n}\]
the functor which sends a Lie algebra $\g$ to the non-unital $\P_n$-algebra $\overline{\Sym}(\g[1-n])$, the reduced symmetric algebra on $\g[1-n]$ with the Poisson bracket induced by the Leibniz rule from the bracket on $\g$. Recall that the symmetric monoidal structure on $\alg_{\P_n}$ we consider is the one transferred under the equivalence $\alg_{\P_n}\cong \alg_{\P_n^{\un}}^{\aug}$ from the usual tensor product of augmented $\P_n$-algebras. In particular, under this equivalence the functor $\overline{\Sym}$ corresponds to $\g\mapsto \Sym(\g[1-n])$ which is symmetric monoidal. The functor $\overline{\Sym}$ preserves weak equivalences, so we obtain a symmetric monoidal functor of $\infty$-categories
\[\overline{\Sym}\colon \ialg_{\Lie}\to \ialg_{\P_n}.\]

\begin{prop}
The diagram
\[
\xymatrix{
\ialg_{\Lie} \ar^{\overline{\Sym}}[d] \ar^-{\sim}[r] & \ialg(\ialg_{\Lie}) \ar^{\overline{\Sym}}[d] \\
\ialg_{\P_{n+1}} \ar[r] & \ialg(\ialg_{\P_n}).
}
\]
is commutative.
\label{prop:LiePoissonSymcompatible}
\end{prop}
\begin{proof}
Let us denote by
\[\Omega_{\Lie}\colon\coalg_{\coComm}\to \alg_{\Lie}\]
the cobar complex for Lie algebras and by
\[\Omega_{\P_n}\colon\coalg_{\coP_n}\to \alg_{\P_n}\]
the cobar complex for $\P_n$-algebras which we can factor as
\[\coalg_{\coP_n}\xrightarrow{\Omega_{\Lie}} \bialg_{\Lie_{n-2}}\to \alg_{\P_n}.\]

Consider the symmetric monoidal functors
\[\triv\colon \coalg_{\coComm}\to \coalg_{\coP_n}\]
and
\[\triv\colon \alg_{\Lie}\to \bialg_{\Lie_{n-2}}\]
which assign trivial cobrackets. We have a commutative diagram
\[
\xymatrix{
\coalg_{\coComm} \ar^-{\Omega_{\Lie}}[r] \ar^{\triv}[d] & \alg_{\Lie} \ar^{\triv}[d] \ar^{\overline{\Sym}}[dr] \\
\coalg_{\coP_n} \ar^-{\Omega_{\Lie}}[r] & \bialg_{\Lie_{n-2}} \ar^-{\Omega}[r] & \alg_{\P_n}
}
\]

Therefore, the claim will follow once we show that the diagram
\[
\xymatrix{
\alg_{\Lie} \ar@{=}[r] \ar^{\overline{\Sym}}[d] \ar^{\overline{\Sym}}[d] & \alg_{\Lie} \ar^-{\U}[r] \ar^{\triv}[d] & \alg(\coalg_{\coComm}) \ar^{\triv}[d] \\
\alg_{\P_{n+1}} \ar[r] & \bialg_{\Lie_{n-1}} \ar^-{\U}[r] & \alg(\coalg_{\coP_n})
}
\]
commutes up to a weak equivalence. It is clear that the square on the right commutes strictly and we are reduced to showing commutativity of the square on the left.

We can factor the functor
\[\overline{\Sym}\colon \alg_{\Lie}\to \alg_{\P_{n+1}}\]
as
\[\alg_{\Lie}\xrightarrow{\triv} \bialg_{\Lie_{n-1}}\xrightarrow{\Omega} \alg_{\P_{n+1}}.\]
Therefore, the composite
\[\alg_{\Lie}\xrightarrow{\triv} \bialg_{\Lie_{n-1}}\xrightarrow{\Omega} \alg_{\P_{n+1}}\xrightarrow{\B} \bialg_{\Lie_{n-1}}\]
is weakly equivalent to $\triv\colon \alg_{\Lie}\rightarrow \bialg_{\Lie_{n-1}}$ and hence the remaining square commutes up to a weak equivalence.
\end{proof}

The functor $\overline{\Sym}\colon \ialg_{\Lie}\rightarrow \ialg_{\P_{n+1}}$ has a right adjoint $\forget\colon \ialg_{\P_{n+1}}\rightarrow \ialg_{\Lie}$ given by forgetting the commutative algebra structure. Since $\overline{\Sym}$ is symmetric monoidal, the right adjoint $\forget$ is lax symmetric monoidal and hence sends associative algebras to associative algebras. Therefore, the commutativity data of Proposition \ref{prop:LiePoissonSymcompatible} gives rise to a diagram of right adjoints
\[
\xymatrix{
\ialg_{\Lie} \ar^-{\sim}[r] \ar@{=>}[dr] & \ialg(\ialg_{\Lie}) \\
\ialg_{\P_{n+1}} \ar[r] \ar^{\forget}[u] & \ialg(\ialg_{\P_n}) \ar_{\forget}[u].
}
\]
which commutes up to a natural transformation.

\begin{prop}
The diagram of right adjoints
\[
\xymatrix{
\ialg_{\Lie} \ar^-{\sim}[r] & \ialg(\ialg_{\Lie}) \\
\ialg_{\P_{n+1}} \ar[r] \ar^{\forget}[u] & \ialg(\ialg_{\P_n}) \ar^{\forget}[u].
}
\]
commutes.
\label{prop:LiePoissonForgetcompatible}
\end{prop}
\begin{proof}
Since the functor $\forget\colon \ialg_{\P_n}\rightarrow \ialg_{\Lie}$ is lax monoidal, we have a commutative diagram
\[
\xymatrix{
\ialg(\ialg_{\Lie}) \ar[r] & \ialg_{\Lie} \\
\ialg(\ialg_{\P_n}) \ar[r] \ar^{\forget}[u] & \ialg_{\P_n} \ar^{\forget}[u]
}
\]
where the horizontal functors are given by forgetting the associative algebra structure.

Forgetting the algebra structure is conservative, so it will be enough to prove that the outer square in
\[
\xymatrix{
\alg_{\Br_{\coP_n}} \ar[r] \ar[d] & \alg_{\Lie} \ar[d] \\
\alg(\coalg_{\coP_n}) \ar[r] \ar[d] & \alg(\coalg_{\coComm}) \ar[d] \\
\coalg_{\coP_n} \ar[r] \ar[d] & \coalg_{\coComm} \ar[d] \\
\alg_{\P_n} \ar^{\forget}[r] & \alg_{\Lie}
}
\]
commutes up to a weak equivalence. That is, suppose $A$ is a $\Br_{\coP_n}$-algebra. Then we have show that the natural morphism of dg Lie algebras
\[\Omega_{\Lie}\B_{\Lie} A\to \Omega_{\P_n}\B_{\P_n} A\]
is a quasi-isomorphism, where $\Omega_{...}$ and $\B_{...}$ are the respective cobar and bar constructions. Let us note that $\B_{\Lie} A$ is the bar complex with respect to the $L_\infty$ structure on $A$ underlying the homotopy $\P_n$-structure on $A$.

But this is a quasi-isomorphism since the diagram
\[
\xymatrix{
\Omega_{\Lie}\B_{\Lie} A \ar[rr] \ar_{\sim}[dr] && \Omega_{\P_n}\B_{\P_n} A \ar^{\sim}[dl] \\
& A
}
\]
commutes.
\end{proof}

We will end this section by proving that the additivity functor for $\P_n$-algebras is an equivalence.

\begin{thm}
The additivity functor
\[\add\colon \ialg_{\P_{n+1}}\to \ialg(\ialg_{\P_n})\]
is an equivalence of symmetric monoidal $\infty$-categories.
\label{thm:Poissonadditivity}
\end{thm}
\begin{proof}
By Proposition \ref{prop:Poissonadditivitysm} the additivity functor is symmetric monoidal, so we just have to show that it is an equivalence of $\infty$-categories.

Consider the diagram
\[
\xymatrix{
\ialg_{\P_{n+1}} \ar^{\forget}[d] \ar[r] & \ialg(\ialg_{\P_n}) \ar^{\forget}[d] \\
\ialg_{\Lie} \ar[r] & \ialg(\ialg_{\Lie})
}
\]

By Corollary \ref{cor:Lieadditivity} the bottom functor is an equivalence. The forgetful functor \[\forget\colon\ialg_{\P_{n+1}}\to \ialg_{\Lie}\] is conservative and preserves sifted colimits since they are created by the forgetful functor to chain complexes by Proposition \ref{prop:forgetsifted}. Similarly, the forgetful functor \[\forget\colon \ialg(\ialg_{\P_n})\to \ialg(\ialg_{\Lie})\] is conservative. Let $\cO = \P_n$ or $\Lie$. Sifted colimits in $\ialg_{\cO}$ are created by the forgetful functor to chain complexes and since sifted colimits in $\ialg(\ialg_{\cO})$ are created by the forgetful functor to $\ialg_{\cO}$, we conclude that sifted colimits in $\ialg(\ialg_{\cO})$ are created by the forgetful functor to chain complexes.

By Proposition \ref{prop:LiePoissonForgetcompatible} the diagram commutes and by Proposition \ref{prop:LiePoissonSymcompatible} the diagram of left adjoints commutes. Therefore, Proposition \ref{prop:luriebarrbeck} applies and the claim follows.
\end{proof}

A natural question is whether the Dunn--Lurie additivity functor (see \cite[Theorem 5.1.2.2]{HA})
\[\ialg_{\E_{n+1}}\stackrel{\sim}\to \ialg(\ialg_{\E_n})\]
is compatible with the Poisson additivity functor in the following sense. Suppose that $n\geq 2$. Then we have an equivalence of Hopf operads $\P_n\cong \E_n$ provided by the formality of the operad $\E_n$ which gives an equivalence of symmetric monoidal $\infty$-categories $\ialg_{\E_n}\cong \ialg_{\P_n}$.

\begin{conjecture}
Suppose $n\geq 2$. Then the diagram
\[
\xymatrix{
\ialg_{\P_{n+1}} \ar^-{\sim}[r] \ar^{\sim}[d] & \ialg(\ialg_{\P_n}) \ar^{\sim}[d] \\
\ialg_{\E_{n+1}} \ar^-{\sim}[r] & \ialg(\ialg_{\E_n})
}
\]
is commutative.
\end{conjecture}

\subsection{Additivity for unital Poisson algebras}
\label{sect:unitalpoissonadditivity}

Let us explain the necessary modifications one performs to construct the additivity functor for unital Poisson algebras. Recall from Section \ref{sect:examples} the bar-cobar adjunction for unital Poisson algebras
\[\adj{\Omega\colon\coalg_{\coP_n^\theta}}{\alg_{\P_n^{\un}}\colon \B},\]
where $\coP_n^\theta$ is the cooperad of curved $\P_n$-coalgebras. One similarly has a bar-cobar adjunction
\[\adj{\Omega\colon\coalg_{\coLie^\theta}}{\alg_{\Comm^{\un}}\colon \B},\]
where $\coLie^\theta$ is the cooperad of curved Lie coalgebras, that is, graded Lie coalgebras $\g$ together with a coderivation $\d$ of degree 1 satisfying
\begin{align*}
\d^2(x) &= \theta(x^{\delta}_{(1)})x^{\delta}_{(2)} \\
\theta(\d x) &= 0,
\end{align*}
where $\delta(x) = x^{\delta}_{(1)}\otimes x^{\delta}_{(2)}$.

The cobar construction is given by the same formula as in the non-curved case except that we add the curving to the differential. The bar construction of a unital commutative algebra $A$ is given by
\[\coLie(A[1]\oplus k[2])\]
with a bar differential and the curving which is given by projecting to $k[2]$.

\begin{defn}
An \defterm{$n$-shifted curved Lie bialgebra} is a curved Lie coalgebra $\g$ of degree $-n$ together with a degree $0$ Lie bracket satisfying the cocycle relation \eqref{eq:Liebialgcocycle}.
\end{defn}

We denote the category of $n$-shifted conilpotent curved Lie bialgebras by $\bialg^\theta_{\Lie_n}$ with weak equivalences created by the forgetful functor
\[\bialg^\theta_{\Lie_n}\to\coalg_{\coLie^\theta}.\]

The commutative bar-cobar adjunction
\[\adj{\Omega\colon\coalg_{\coLie^\theta}}{\alg_{\Comm^{\un}}\colon \B},\]
extends to a bar-cobar adjunction
\[\adj{\Omega\colon \bialg^\theta_{\Lie_{n-1}}}{\alg_{\P_{n+1}^{\un}}\colon \B}.\]
The only modification from the case of non-unital Poisson algebras is the formula for the Lie bracket. Suppose $A$ is a unital $\P_{n+1}$-algebra. Then as a graded vector space
\[\B(A) = \coLie(A[1]\oplus k[2])[n-1].\]
By the cocycle equation \eqref{eq:Liebialgcocycle}, the Lie bracket on $\coLie(A[1]\oplus k[2])$ is uniquely determined after projection to cogenerators and the morphism
\begin{align*}
\coLie(A[1]\oplus k[2])[n-1]\otimes \coLie(A[1]\oplus k[2])[n-1]&\to \coLie(A[1]\oplus k[2])[n-1]\\
&\to A[n]\oplus k[n+1]
\end{align*}
has the zero component in $k[n+1]$ and its $A[n]$ component is defined to be the bracket
\[A[n]\otimes A[n]\to A[n].\]

The universal enveloping algebra construction gives a functor
\[\U\colon \bialg^\theta_{\Lie_{n-1}}\to \alg(\coalg_{\coP_n^\theta}).\]
Therefore, we can define the additivity functor
\[\add_{\P_n^{\un}}\colon \ialg_{\P_{n+1}^{\un}}\to \ialg(\ialg_{\P_n^{\un}})\]
to be given by the composite
\begin{align*}
\alg_{\P_{n+1}^{\un}}[\qis^{-1}]&\to \bialg^\theta_{\Lie_{n_1}} [\kos^{-1}] \\
&\to \alg(\coalg_{\coP_n^\theta})[\kos^{-1}] \\
&\to \ialg(\coalg_{\coP_n^\theta}[\kos^{-1}]).
\end{align*}

The following statement is proved as for non-unital Poisson algebras by analyzing the forgetful functor to Lie algebras.
\begin{thm}
The additivity functor
\[\ialg_{\P_{n+1}^{\un}}\to \ialg(\ialg_{\P_n^{\un}})\]
is an equivalence of symmetric monoidal $\infty$-categories.
\end{thm}

\section{Coisotropic structures}

In this section we show that two definitions of derived coisotropic structures given in \cite{CPTVV} and \cite{MS} are equivalent. Both definitions are given first in the affine setting and then extended in the same way to derived stacks, so it will be enough to prove equivalence on the affine level.

\subsection{Two definitions}

Let us introduce the following notations. Given a category $\cC$ we denote by $\Arr(\cC)$ the category of morphisms in $\cC$, i.e. the functor category $\Fun(\Delta^1, \cC)$. Given a symmetric monoidal category $\cC$ we denote by $\lmod(\cC)$ the category of pairs $(A, M)$, where $A$ is a unital associative algebra in $\cC$ and $M$ is a left $A$-module. Let us denote by $\ilmod(\cC)$ the same construction for a symmetric monoidal $\infty$-category $\cC$.

Introduce the notation
\[\ialg_{\P_{(n+1, n)}} = \ilmod(\ialg_{\P_n}).\]
We have a forgetful functor $\ialg_{\P_{(n+1, n)}}\rightarrow \Arr(\ialg_{\Comm})$ which sends a pair $(A, B)$ with the action map $A\otimes B\rightarrow B$ to the morphism of commutative algebras $A\rightarrow B$ given by the composite $A\stackrel{\id\otimes 1}\rightarrow A\otimes B\rightarrow B$.

The following definition of coisotropic structures is due to Calaque--Pantev--To\"{e}n--Vaqui\'{e}--Vezzosi \cite[Definition 3.4.3]{CPTVV}:
\begin{defn}
Let $f\colon A\rightarrow B$ be a morphism of commutative dg algebras. The \defterm{space of $n$-shifted coisotropic structures} $\Cois^{CPTVV}(f, n)$ is defined to be the fiber of
\[\ialg_{\P_{(n+1, n)}}\to \Arr(\ialg_{\Comm})\]
at the given morphism $f$.
\end{defn}

To relate this space of $n$-shifted coisotropic structures to the space of $n$-shifted Poisson structures on $A$, one has to use the Poisson additivity functor
\[\add\colon \ialg_{\P_{n+1}}\to \ialg(\ialg_{\P_n}).\]

A more explicit definition of the space of coisotropic structures was given in \cite{MS} and \cite{Sa} as follows. Let $B$ be a $\P_n$-algebra. We define its \defterm{strict center} to be the $\P_{n+1}$-algebra
\[\Z^{\str}(B) = \Hom_B(\Sym_B(\Omega^1_B[n]), B)\]
with the differential twisted by $[\pi_B, -]$. We refer to \cite[Section 1.1]{Sa} for explicit formulas for the $\P_{n+1}$-structure. One can also consider its center
\[\Z(B) = \Hom_k(\coP^{\cu}_n(B), B)[-n]\]
with the differential twisted by the $\P_n$-structure on $B$. By results of \cite{CW}, $\Z(B)$ is a $\Br_{\coP_n}\{n\}$-algebra and hence in particular a homotopy $\P_{n+1}$-algebra. Moreover, the natural inclusion
\[\Z^{\str}(B)\rightarrow \Z(B)\]
is compatible with the homotopy $\P_{n+1}$ structures on both sides and is a quasi-isomorphism if $B$ is cofibrant as a commutative algebra.

Let $\P_{[n+1, n]}$ be the colored operad whose algebras consist of a pair $(A, B, f)$ where $A$ is a $\P_{n+1}$-algebra, $B$ is a $\P_n$-algebra and $f\colon A\rightarrow \Z^{\str}(B)$ is a $\P_{n+1}$-morphism. The projection $\Z^{\str}(B)\rightarrow B$ is a morphism of commutative algebras, so we get a natural forgetful functor
\[\alg_{\P_{[n+1, n]}}\rightarrow \Arr(\alg_{\Comm}).\]

\begin{defn}
Let $f\colon A\rightarrow B$ be a morphism of commutative dg algebras. The \defterm{space of $n$-shifted coisotropic structures} $\Cois^{MS}(f, n)$ is defined to be the fiber of
\[\ialg_{\P_{[n+1, n]}}\to \Arr(\ialg_{\Comm})\]
at the given morphism $f$.
\end{defn}

Our goal will be to construct an equivalence of $\infty$-categories $\ialg_{\P_{[n+1, n]}}\rightarrow \ialg_{\P_{(n+1, n)}}$ which is compatible with the forgetful functor to $\Arr(\ialg_{\Comm})$ which will show that the spaces $\Cois^{MS}(f, n)$ and $\Cois^{CPTVV}(f, n)$ are equivalent. We will construct the equivalence as a relative version of the additivity functor \eqref{eq:additivityfunctor}.

\subsection{Relative additivity for Lie algebras}

We begin with the relative analog of the additivity functor for Lie algebras.

Let us introduce a Swiss-cheese operad of Lie algebras. Given a dg Lie algebra $\h$ the Chevalley--Eilenberg complex $\C^\bullet(\h, \h)[1]$ carries a convolution Lie bracket. We let $\Lie_{[1, 0]}$ be the colored operad whose algebras consist of a pair of dg Lie algebras $(\g, \h)$ together with a map of dg Lie algebras $\g\rightarrow \C^\bullet(\h, \h)[1]$.

\begin{remark}
Following \cite[Section 3.3]{MS}, given the morphism of operads
\[\Omega \coComm\{1\}\rightarrow \Lie\rightarrow \Br_{\coComm},\]
one can consider the colored operad $\mathrm{SC}(\coComm, \coComm)$ whose algebras are given by a pair $(\g, \h)$ of $L_\infty$ algebras together with an $\infty$-morphism $\g\rightarrow \C^\bullet(\h, \h)[1]$. The colored operad $\Lie_{[1, 0]}$ is given by a quotient where we declare both $\g$ and $\h$ to be strict Lie algebras and the morphism $\g\rightarrow \C^\bullet(\h, \h)[1]$ to be a strict morphism.
\end{remark}

Consider the morphism $\h\rightarrow \Der(\h)$ given by the adjoint action and denote by $\widetilde{\Der}(\h)$ its cone. Note that we have an obvious inclusion $\widetilde{\Der}(\h)\subset \C^\bullet(\h, \h)[1]$. We denote by $\Lie^{\str}_{[1, 0]}$ the quotient operad of $\Lie_{[1, 0]}$ where we set all components of the map $\g\rightarrow \C^\bullet(\h, \h)[1]$ having $\h$-arity at least $2$ to be zero. Thus, a $\Lie^{\str}_{[1, 0]}$-algebra is a pair $(\g, \h)$ of dg Lie algebras together with a map of Lie algebras $\g\rightarrow \widetilde{\Der}(\h)$.

Given a pair $(\g, \h)\in \alg_{\Lie^{\str}_{[1, 0]}}$ we can construct a dg Lie algebra $\k$ as follows. As a graded vector space we define $\k = \g\oplus \h$. The differential on $\k$ is the sum of the differentials on $\g$ and $\h$ and the differential $\g\rightarrow \h$ coming from the composite
\[\g\to \C^\bullet(\h, \h)[1]\to \h[1].\]
The Lie bracket on $\k$ is the sum of Lie brackets on $\g$ and $\h$ and the action map of $\g$ on $\h$ given by the map $\g\rightarrow \Der(\h)$. In this way we see that a pair $(\g, \h)\in \alg_{\Lie^{\str}_{[1, 0]}}$ is the same as a dg Lie algebra extension
\[0\to \h\to \k\to \g\to 0.\]
One can similarly construct an $L_\infty$ extension from the data of a general object of $\alg_{\Lie_{[1, 0]}}$. Since every $L_\infty$ algebra is quasi-isomorphic to a dg Lie algebra, the following statement should be obvious.

\begin{prop}
The forgetful functor
\[\alg_{\Lie^{\str}_{[1, 0]}}\to \alg_{\Lie_{[1, 0]}}\]
induces an equivalence on the underlying $\infty$-categories.
\end{prop}
\begin{proof}
By Proposition \ref{prop:operadforgetful} it is enough to show that the projection $\Lie_{[1, 0]}\rightarrow \Lie^{\str}_{[1, 0]}$ is a quasi-isomorphism.

Let $\coComm^{\cu}$ be the cooperad of counital cocommutative coalgebras. One can identify the underling graded colored symmetric sequence $\Lie_{[1, 0]}$ with $\Lie\circ (\coComm^{\cu}\circ\Lie\oplus \bu)$ so that $\Lie_{[1, 0]}(\cA^{\otimes n}, \cB^{\otimes m})$ has $n$ operations in $\coComm^{\cu}$ and $m$ operations in $\Lie$ or $\bu$. The projection \[\Lie_{[1, 0]}(\cA^{\otimes 0}, \cB^{\otimes m})\rightarrow \Lie^{\str}(\cA^{\otimes 0}, \cB^{\otimes m})\] is an isomorphism. To simplify the notation, we prove that the projection
\begin{equation}
\Lie_{[1, 0]}(\cA^{\otimes 1}, \cB^{\otimes m})\rightarrow \Lie^{\str}(\cA^{\otimes 1}, \cB^{\otimes m})
\label{eq:liestrprojection}
\end{equation}
is a quasi-isomorphism since the case of higher $n$ is handled similarly. Thus, we can identify the underlying graded symmetric sequence $\Lie_{[1, 0]}(\cA^{\otimes 1}, \cB^{\otimes -})$ wth $\Lie\circ_{(1)} (\coComm^{\cu}\circ \Lie)$, where $\circ_{(1)}$ is the infinitesimal composite (see \cite[Section 6.1]{LV}).

Let $V$ be a complex and consider $\Lie(V)$, the free Lie algebra on $V$, and its homology $\C_\bullet(\Lie(V), \Lie(V))$. Let $\C_{\leq 1}(\Lie(V), \Lie(V))$ be the quotient of $\C_\bullet(\Lie(V), \Lie(V))$ where we consider only components of weight $\leq 1$ and mod out by the image of the Chevalley--Eilenberg differential from weight $2$ to weight $1$. The projection
\[\C_\bullet(\Lie(V), \Lie(V))\rightarrow \C_{\leq 1}(\Lie(V), \Lie(V))\]
is a quasi-isomorphism which can be seen as follows. The left-hand side by definition computes the derived tensor product $k\otimes^{\bL}_{\T(V)} \Lie(V)$. But $k$ has a two-term free resolution
\[(\T(V)\otimes V\rightarrow \T(V))\stackrel{\sim}\rightarrow k\]
as a $\T(V)$-module and computing the derived tensor product using this resolution one exactly obtains $\C_{\leq 1}(\Lie(V), \Lie(V))$.

To conclude the proof, observe that the coefficient of $V^{\otimes m}$ in $\C_\bullet(\Lie(V), \Lie(V))$ is isomorphic as a complex to $\Lie_{[1, 0]}(\cA^{\otimes 1}, \cB^{\otimes m})$ while its coefficient in $\C_{\leq 1}(\Lie(V), \Lie(V))$ is isomorphic to $\Lie^{\str}(\cA^{\otimes 1}, \cB^{\otimes m})$.
\end{proof}

Now we are going to introduce the bar construction
\[\B\colon \alg_{\Lie_{[1, 0]}}\to \lmod(\coalg_{\coComm})\]
Consider a pair $(\g, \h)\in \alg_{\Lie_{[1, 0]}}$. We send it to the pair $(\U(\g), \C_\bullet(\h))$ of a cocommutative bialgebra and a cocommutative coalgebra. The action map
\[\U\g\otimes \C_\bullet(\h)\to \C_\bullet(\h)\]
is constructed as follows. Since $\U(\g)$ is generated by $\g$, it is enough to specify the action $\g\otimes \C_\bullet(\h)\rightarrow \C_\bullet(\h)$ that we denote by $x. c$ for $x\in\g$ and $c\in\C_\bullet(\h)$ satisfying the equations
\begin{align*}
(x.c)_{(1)}\otimes (x.c)_{(2)} &= (x.c_{(1)})\otimes c_{(2)} + (-1)^{|x||c_{(1)}|}c_{(1)}\otimes (x.c_{(2)}),\qquad x\in\g, c\in\C_\bullet(\h) \\
[x,y].c &= x.(y.c) - (-1)^{|x||y|}y.(x.c) \qquad x,y\in\g, c\in\C_\bullet(\h)
\end{align*}

Since $\C_\bullet(\h)$ is cofree, by the first equation it is enough to specify the map $\g\otimes \C_\bullet(\h)\rightarrow \h[1]$ which we define to be adjoint to the given map $\g\rightarrow \C^\bullet(\h, \h)[1]$. It is then easy to see that since the map $\g\rightarrow \C^\bullet(\h, \h)[1]$ is compatible with Lie brackets, the second equation is satisfied and since it is compatible with the differentials, so is the map $\g\otimes \C_\bullet(\h)\rightarrow \C_\bullet(\h)$. This defines the functor
\[\B\colon \alg_{\Lie_{[1, 0]}}\to \lmod(\coalg_{\coComm}).\]

We can also introduce the cobar construction
\[\Omega\colon \lmod(\coalg_{\coComm})\to \alg_{\Lie^{\str}_{[1, 0]}}.\]
Consider an element $(A, C)\in \lmod(\coalg_{\coComm})$ where $A$ is an algebra and $C$ is an $A$-module. By Theorem \ref{thm:CartierMilnorMoore} we can identify $A\cong \U(\g)$ for the Lie algebra $\g$ of primitive elements. We send $(\U(\g), C)$ to the pair $(\g, \h=\Omega C)$, where $\Omega C$ is the Harrison complex $\Lie(\overline{C}[-1])$. Let us denote the action map $\U\g\otimes C\rightarrow C$ by $x.c$ for $x\in\g$ and $c\in C$. The morphism $\g\to\h[1]$ is defined by the composite
\[\g\to \overline{C}[-1]\to \Lie(\overline{C}[-1])=\h\]
where the first map is given by the action map $x.1$. Since $\h$ is semi-free, the morphism $\g\rightarrow \Der(\h)$ is uniquely determined by the map
\[\g\otimes \overline{C}[-1]\to \overline{C}[-1]\to \Lie(\overline{C}[-1]),\]
where the first map is the action of $\g$ on $C$. The fact that thus constructed morphism $\g\rightarrow \widetilde{\Der}(\h)$ is a morphism of Lie algebras follows from the associativity of the action map $\U\g\otimes C\rightarrow C$. The compatibility of the morphism $\g\rightarrow \widetilde{\Der}(\h)$ with the differential follows from the compatibility of the action map $\U(\g)\otimes C\rightarrow C$ with coproducts. This defines the functor
\[\Omega\colon \lmod(\coalg_{\coComm})\to \alg_{\Lie^{\str}_{[1, 0]}}.\]

Note that in this way we obtain an adjunction
\[\adj{\Omega\colon \lmod(\coalg_{\coComm})}{\alg^{\str}_{\Lie_{[1, 0]}}\colon \B}.\]
Indeed, the counit and unit morphisms
\[\Omega\B(\g, \h)\to (\g, \h)\qquad (A, C)\rightarrow \B\Omega(A, C)\]
are defined to be the identities in the first slot and the counit and unit of the usual bar-cobar adjunction in the second slot.

\begin{prop}
The adjunction
\[\adj{\Omega\colon \lmod(\coalg_{\coComm})}{\alg_{\Lie^{\str}_{[1, 0]}}\colon \B}\]
induces an adjoint equivalence
\[\adj{\lmod(\coalg_{\coComm})[\kos^{-1}]}{\alg_{\Lie^{\str}_{[1, 0]}}[\qis^{-1}]}\]
on the underlying $\infty$-categories.
\label{prop:relLieadditivityB}
\end{prop}
\begin{proof}
We have a commutative diagram
\[
\xymatrix{
\lmod(\coalg_{\coComm}) \ar[d] \ar@<.5ex>[r] & \alg_{\Lie^{\str}_{[1, 0]}} \ar[d] \ar@<.5ex>[l] \\
\alg(\coalg_{\coComm})\times \coalg_{\coComm} \ar@<.5ex>[r] & \alg_{\Lie}\times \alg_{\Lie} \ar@<.5ex>[l]
}
\]
where the vertical functors are the obvious forgetful functors. Since they reflect weak equivalences, it is enough to show the bottom adjunction induces an equivalence after localization. Indeed,
\[\xymatrix{
\alg(\coalg_{\coComm}) \ar@<.5ex>[r] & \alg_{\Lie} \ar@<.5ex>^-{\U}[l]
}\]
is an equivalence by Theorem \ref{thm:CartierMilnorMoore} and
\[\adj{\coalg_{\coComm}}{\alg_{\Lie}}\]
induces an $\infty$-categorical equivalence by Proposition \ref{prop:koszulduality}.
\end{proof}

The composite functor
\[\alg_{\Lie_{[1, 0]}} [\qis^{-1}]\stackrel{\B}\to \lmod(\coalg_{\coComm})[\kos^{-1}]\to \ilmod(\coalg_{\coComm}[\kos^{-1}])\]
defines a relative additivity functor for Lie algebras:
\[\add\colon \ialg_{\Lie_{[1, 0]}}\to \ilmod(\ialg_{\Lie}).\]

\begin{thm}
The additivity functor
\[\ialg_{\Lie_{[1, 0]}}\to \ilmod(\ialg_{\Lie})\]
is an equivalence of $\infty$-categories.
\label{thm:relLieadditivity}
\end{thm}
\begin{proof}
Consider a commutative diagram
\[
\xymatrix{
\ialg_{\Lie_{[1, 0]}} \ar^{\B}[d] & \\
\lmod(\coalg_{\coComm})[\kos^{-1}] \ar^-{G_1}[r] \ar^{L_1}[d] & \alg(\coalg_{\coComm})[\kos^{-1}]\times \coalg_{\coComm}[\kos^{-1}] \ar^{L_2}[d] \\
\ilmod(\coalg_{\coComm}[\kos^{-1}]) \ar^-{G_2}[r] & \ialg(\coalg_{\coComm}[\kos^{-1}])\times \coalg_{\coComm}[\kos^{-1}]
}
\]
where the functors $G_1,G_2$ are the obvious forgetful functors.

The functor
\[\B\colon \ialg_{\Lie_{[1, 0]}}\to \lmod(\coalg_{\coComm})[\kos^{-1}]\]
is an equivalence by Proposition \ref{prop:relLieadditivityB}, so we just need to show that the functor $L_1$ is an equivalence. The localization functor
\[\alg(\coalg_{\coComm})[\kos^{-1}]\to \ialg(\coalg_{\coComm}[\kos^{-1}])\]
is an equivalence by results of Section \ref{sect:Lieadditivity}, so the functor $L_2$ is an equivalence.

The functor $G_1$ is conservative. Sifted colimits in $\ialg_{\Lie_{[1, 0]}}$ are created by the forgetful functor to $\iCh\times\iCh$ by Proposition \ref{prop:forgetsifted}, so the functor $G_1$ preserves sifted colimits. The functor $G_2$ is conservative; it preserves sifted colimits by \cite[Corollay 4.2.3.7]{HA}.

The left adjoint to $G_1$ post-composed with the forgetful functor to $\coalg_{\coComm}[\kos^{-1}]$ is given by the functor $(A, V)\mapsto (A, A\otimes V)$. By \cite[Corollary 4.2.4.4]{HA} the functor $G_2$ also admits a left adjoint given by the same formula, so the diagram of left adjoints commutes. Therefore, Proposition \ref{prop:luriebarrbeck} applies and thus $L_1$ is an equivalence.
\end{proof}

\begin{remark}
One can also construct the relative additivity functor $\ialg_{\Lie_{[1, 0]}}\cong \ilmod(\ialg_{\Lie})$ as follows. The colored operad $\Lie^{\str}_{[1, 0]}$ is quadratic whose Koszul dual is the cooperad of a pair of cocommutative coalgebras $C_1,C_2$ together with a morphism $C_1\rightarrow C_2$. Finally, a relative version of Proposition \ref{prop:Lieloopequivalence} gives an equivalence $\Arr(\ialg_{\Lie})\cong \ilmod(\ialg_{\Lie})$.
\end{remark}

\subsection{Relative additivity for Poisson algebras}

We now proceed to the construction of the additivity functor
\[\add\colon \ialg_{\P_{[n+1, n]}}\to \ilmod(\alg_{\P_n}).\]

Consider a pair $(A,B)\in\alg_{\P_{[n+1, n]}}$. Let $\g$ be the Koszul dual $(n-1)$-shifted Lie bialgebra to $A$ constructed in Section \ref{sect:poissonadditivity}. As a graded vector space, we can identify $\g\cong\coLie(A[1])[n-1]$. Let us also denote by $\B B$ the Koszul dual coaugmented $\P_n$-coalgebra; as a graded vector space, we can identify $\B B\cong \coP^{\cu}_n(B[n])$. Recall also that $\U(\g)$ is an associative algebra in $\P_n$-coalgebras. Now we want to construct the action map
\[a\colon \U(\g)\otimes \B B\to \B B\]
of $\P_n$-coalgebras. Such a map by associativity is uniquely determined by the map
\[\g\otimes \B B\to \B B\]
and since $\B B$ is cofree as a graded $\P_n$-coalgebra, this map is uniquely determined by projection to the cogenerators
\[\g\otimes \B B\to B[n].\]
We define this map to be adjoint to the map
\[\coLie(A[1])[n-1]\to A[n]\to \Z(B)[n].\]

Recall the description of the Koszul dual Lie bialgebra to a $\Br_{\coP_n}$-algebra such as $\Z(B)$ from the proof of Proposition \ref{prop:bracepoissoncompatible}. Using this description, we see that the associativity of the action map $a$ follows from the compatibility of the morphism $A\to \Z(B)$ with Lie brackets and the compatibility of $a$ with the differential follows from the compatibility of $A\to \Z(B)$ with the $C_\infty$ structure. Compatibility with the $\P_n$-coalgebra structures is obvious by construction.

In this way we obtain a functor
\[\alg_{\P_{[n+1, n]}}\to \lmod(\coalg_{\P_n})\]
and we let the additivity functor
\[\ialg_{\P_{[n+1, n]}}\to \ilmod(\coalg_{\P_n}[\kos^{-1}])\]
be the composite
\[\alg_{\P_{[n+1, n]}}[\qis^{-1}]\to \lmod(\coalg_{\P_n})[\kos^{-1}]\to \ilmod(\coalg_{\P_n}[\kos^{-1}]).\]

\begin{thm}
The additivity functor
\[\add\colon \ialg_{\P_{[n+1, n]}}\to \ilmod(\ialg_{\P_n})\]
is an equivalence of $\infty$-categories.
\label{thm:relPoissonadditivity}
\end{thm}
\begin{proof}
Suppose $B$ is a $\P_n$-algebra. In particular, $B[n-1]$ is a Lie algebra and we have a morphism
\[\Z(B)[n]\to \C^\bullet(B[n-1], B[n])\]
of Lie algebras induced by the morphism of cooperads $\coComm\rightarrow \coP_n$. Compatibility with the Lie algebra structures on both sides is clear as both are given by convolution brackets.

This gives a forgetful functor
\[G_1\colon \alg_{\P_{[n+1, n]}}\to \alg_{\Lie_{[1, 0]}}\]
which sends a pair $(A, B)$ with a morphism of $\P_{n+1}$-algebras $A\to \Z^{\str}(B)$ to the pair of Lie algebras $(A[n], B[n-1])$ with a morphism of Lie algebras
\[A[n]\to\Z^{\str}(B)[n]\to \C^\bullet(B[n-1], B[n-1])[1].\]

The forgetful functor $G_1$ has a left adjoint
\[F_1\colon \alg_{\Lie_{[1, 0]}}\to \alg_{\P^{\str}_{[n+1, n]}}\]
which is constructed as follows. Consider a pair $(\g, \h)\in\alg_{\Lie_{[1, 0]}}$ equipped with a morphism $\g\rightarrow \C^\bullet(\h, \h)[1]$. Then $A=\overline{\Sym}(\g[-n])$ is a $\P_{n+1}$-algebra and $B=\overline{\Sym}(\h[1-n])$ is a $\P_n$-algebra. We can identify
\[\Z^{\str}(B)\cong \C^\bullet(\h, B)\]
as Lie algebras. Using this identification we obtain a morphism of $\P_n$-algebras
\[A\to\Z^{\str}(B)\]
defined to be the Lie map $\g\rightarrow\C^\bullet(\h, \h)[1]$ on the generators of $A$. This concludes the construction of the functor $F_1$.

Now consider the diagram
\begin{equation}
\xymatrix{
\alg_{\P^{\str}_{[n+1, n]}} \ar[r] & \lmod(\coalg_{\P_n}) \\
\alg_{\Lie_{[1, 0]}} \ar^-{\B}[r] \ar^{F_1}[u] & \lmod(\coalg_{\coComm}). \ar^{\triv}[u]
}
\label{eq:relPoissondiagram1}
\end{equation}
Consider an object $(\g,\h)\in\alg_{\Lie_{[1, 0]}}$. Under the composite
\[\alg_{\Lie_{[1, 0]}}\to \alg_{\P^{\str}_{[n+1, n]}}\to \lmod(\coalg_{\P_n})\]
the underlying associative algebra can be identified with $\U\g$ with the trivial cobracket by Proposition \ref{prop:LiePoissonSymcompatible}. Similarly, the underlying module can be identified with $\C_\bullet(\h)$ with the trivial cobracket using the weak equivalence \[\C_\bullet(\h)\stackrel{\sim}\rightarrow \B_{\P_n}(\overline{\Sym}(\h[1-n])).\]
It is easy to see that the action of $\U\g$ on $\C_\bullet(\h)$ under this equivalence coincides with the action given by the composite
\[\alg_{\Lie_{[1, 0]}}\to \lmod(\coalg_{\coComm})\to\lmod(\coalg_{\P_n})\]
and hence the diagram \eqref{eq:relPoissondiagram1} commutes up to a weak equivalence.

Denote by
\[\adj{F_2\colon \ialg_{\Lie}}{\ialg_{\P_n}\colon G_2}\]
the free-forgetful adjunction and consider a diagram of $\infty$-categories
\begin{equation}
\xymatrix{
\ialg_{\P_{[n+1, n]}} \ar^-{\add_{\P_n}}[r] \ar@<.5ex>^{G_1}[d] & \ilmod(\ialg_{\P_n}) \ar@<.5ex>^{G_2}[d] \\
\ialg_{\Lie_{[1, 0]}} \ar^-{\add_{\Lie}}[r] \ar@<.5ex>^{F_1}[u] & \ilmod(\ialg_{\Lie}) \ar@<.5ex>^{F_2}[u]
}
\label{eq:relPoissondiagram2}
\end{equation}
By Theorem \ref{thm:relLieadditivity} the functor $\add_{\Lie}$ is an equivalence. Moreover, the commutativity of diagram \eqref{eq:relPoissondiagram1} implies that
\[\add_{\P_n}\circ F_1\cong F_2\circ\add_{\Lie}.\]
To check that the natural morphism
\[\add_{\Lie}\circ G_1\rightarrow G_2\circ \add_{\P_n}\]
is an equivalence, it is enough to check it in $\ialg(\ialg_{\Lie})\times \ialg_{\Lie}$ since the forgetful functor
\[\ilmod(\ialg_{\Lie})\to \ialg(\ialg_{\Lie})\times \ialg_{\Lie}\]
is conservative. By Proposition \ref{prop:LiePoissonForgetcompatible} the corresponding morphism in $\ialg(\ialg_{\Lie})$ is an equivalence. It is also obvious that the corresponding morphism in $\ialg_{\Lie}$ is an equivalence since the diagram \eqref{eq:relPoissondiagram2} can be forgotten to the commutative diagram
\[
\xymatrix{
\ialg_{\P_n} \ar^{\id}[r] \ar@<.5ex>^{\forget}[d] & \ialg_{\P_n} \ar@<.5ex>^{\forget}[d] \\
\ialg_{\Lie} \ar^{\id}[r] \ar@<.5ex>^{\overline{\Sym}}[u] & \ialg_{\Lie} \ar@<.5ex>^{\overline{\Sym}}[u]
}
\]

The forgetful functor $G_1$ is conservative and it preserves sifted colimits since they are created by the forgetful functor to $\iCh\times\iCh$ by Proposition \ref{prop:forgetsifted}.  Similarly, the forgetful functor $G_2$ is conservative and preserves sifted colimits since sifted colimits in $\ilmod(\ialg_{\cO})$ are created by the forgetful functor to $\ialg_{\cO}\times\ialg_{\cO}$ and hence by the forgetful functor to $\iCh\times \iCh$. Therefore, by Proposition \ref{prop:luriebarrbeck} the functor $\add_{\P_n}$ is an equivalence.
\end{proof}

\begin{cor}
Given a morphism $f\colon A\rightarrow B$ of dg commutative algebras, there is a canonical equivalence of spaces of $n$-shifted coisotropic structures
\[\Cois^{MS}(f, n)\stackrel{\sim}\to \Cois^{CPTVV}(f, n).\]
\label{cor:coisotropicadditivity}
\end{cor}
\begin{proof}
To prove the claim we have to show that we have a commutative diagram of $\infty$-categories
\[
\xymatrix{
\ialg_{\P_{[n+1, n]}} \ar^{\add_{\P_n}}[rr] \ar[dr] && \ilmod(\ialg_{\P_n}) \ar[dl] \\
& \Arr(\ialg_{\Comm}) &
}
\]

The forgetful functor $\alg_{\P_n}\rightarrow \alg_{\Comm}$ under Koszul duality corresponds to the functor $\coalg_{\coP_n}\rightarrow \coalg_{\coLie}$ which is given by taking primitive elements. We are going to show that the diagram
\[
\xymatrix{
\alg_{\P_{[n+1, n]}} \ar[r] \ar[dd] & \lmod(\coalg_{\coP_n}) \ar[d] \\
& \Arr(\coalg_{\coP_n}) \ar[d] \\
\Arr(\alg_{\Comm}) \ar[r] & \Arr(\coalg_{\coLie})
}
\]
strictly commutes which will prove the claim.

Consider an object $(A, B)\in\alg_{\P_{[n+1, n]}}$. Let $\g$ be the $(n-1)$-shifted Lie bialgebra Koszul dual to $A$ and $\B_{\P_n} B$ the coaugmented $\P_n$-coalgebra Koszul dual to $B$. The functor
\[\lmod(\coalg_{\coP_n})\to \Arr(\coalg^{\coaug}_{\coP^{\cu}_n})\]
sends the action map $\U\g\otimes \B_{\P_n} B\rightarrow \B_{\P_n} B$ to the morphism $\U\g\rightarrow \B_{\P_n} B$ which is given by the image of $1\in B$. After passing to primitives we obtain a morphism
\begin{equation}
\g\to \B_{\Comm} B
\label{eq:coisotropicadditivitymor}
\end{equation}
of Lie coalgebras. But as a Lie coalgebra we can identify $\g\cong \B_{\Comm} A$ and the morphism \eqref{eq:coisotropicadditivitymor} is the image of $A\rightarrow B$ under the commutative bar construction.
\end{proof}

\begin{remark}
In \cite{MS} we construct a forgetful map from $n$-shifted Poisson structures to $(n-1)$-shifted Poisson structures on a commutative algebra $A$ using the natural correspondence
\[\xymatrix{
\mathrm{Pois}(A, n-1) & \Cois^{MS}(\id, n) \ar[l] \ar^{\sim}[r] & \mathrm{Pois}(A, n).
}\]
We can relate it to the additivity functor as follows. Let $A$ be a $\P_{n+1}$-algebra and $\g$ the Koszul dual $(n-1)$-shifted Lie bialgebra. Then $\U(\g)\in\alg(\coalg_{\coP_n})$ is naturally a module over itself which by Corollary \ref{cor:coisotropicadditivity} gives a coisotropic structure on the identity $A\rightarrow A$, i.e. an element of $\Cois^{MS}(\id, n)$. The underlying $\P_n$-algebra structure in $\mathrm{Pois}(A, n-1)$ is then the Koszul dual $\P_n$-algebra to the $\P_n$-coalgebra $\U(\g)$. But this exactly coincides with the forgetful functor
\[\xymatrix{
\ialg_{\P_{n+1}} \ar^{\sim}[r] & \ialg(\ialg_{\P_n}) \ar[r] & \ialg_{\P_n}.
}
\]
\end{remark}

\printbibliography

\end{document}